\documentclass[12pt]{amsart}

\usepackage[english]{babel}
\usepackage{amsmath,amssymb,amsthm}
\usepackage{amsfonts}
\usepackage{hyperref}
\usepackage[centering]{geometry}
\usepackage{tikz-cd}
\usepackage{pb-diagram}

\usepackage{verbatim,amscd,amsmath,amsthm,amstext,amssymb,amsfonts,latexsym,lscape,rawfonts}

\usepackage[arrow,matrix,curve]{xy}
\usepackage{xcolor}

\newcommand{\di}{\operatorname{diag}}
\newcommand{\tr}{\operatorname{tr}}

\newtheorem{theorem}{Theorem}[section]
\newtheorem{corollary}[theorem]{Corollary}
\newtheorem{lemma}[theorem]{Lemma}
\newtheorem{proposition}[theorem]{Proposition}
\theoremstyle{definition}
\newtheorem{definition}{Definition}

\newtheorem{remark}{Remark}
\newtheorem*{acknow}{Acknowledgments}

\newtheorem*{example}{Example}

\numberwithin{equation}{section}

\newcommand{\R}{\mathbb{R}}
\newcommand{\C}{\mathbb{C}}
\newcommand{\D}{\mathbb{D}}
\newcommand{\E}{\mathbb{E}}

\newcommand{\St}{\mathbb{S}}
\newcommand{\F}{\mathbb{F}}
\newcommand{\Z}{\mathbb{Z}}

\newcommand{\f}{\mathfrak f}

\newcommand{\SU}{{\rm SU}(3)}

\begin{document}

\title[Minimal Lagrangian surfaces in $\mathbb{C}P^2$ : Part I]
{Minimal Lagrangian surfaces in $\mathbb{C}P^2$ via the loop group method   Part  I: the contractible case}
\author{Josef F.~Dorfmeister}
\address{Fakult\"{a}t F\"{u}r Mathematik, TU-M\"{u}nchen, Boltzmann Str. 3,
D-85747, Garching, Germany}
\email{dorfm@ma.tum.de}
\author{Hui Ma}
\address{Department of Mathematical Sciences, Tsinghua University,
Beijing 100084, P.R. China} \email{ma-h@tsinghua.edu.cn}

\maketitle

\begin{abstract} In this paper, we employ  the loop group method to study the construction of minimal Lagrangian surfaces in the complex projective plane for which the surface is contractible. We present several new classes of minimal Lagrangian surfaces in $\C P^2$.
\end{abstract}

\section{Introduction}

Minimal Lagrangian surfaces in the complex projective plane $\mathbb{C}P^2$ endowed with the Fubini-Study metric are of great interest from the point of view of differential geometry, symplectic geometry and mathematical physics
(For instance, see \cite{CU94, Sharipov, MM, Haskins_Am04, McIn, Mironov, Haskins_Inv04}). 
They give rise to  local models of singular special Lagrangian $3$-folds in Calabi-Yau 3-folds, hence play an important role in the development of mirror symmetry (\cite{Joyce}). 

The Gauss-Codazzi equations for minimal Lagrangian surfaces in $\mathbb{C}P^2$
are given by
\begin{equation*}\label{eq:mLi}
\begin{split}
u_{z\bar z}&=e^{-2u}|\psi|^2-e^{u},\\
\psi_{\bar z}&=0,
\end{split}
\end{equation*}
where $g=2e^{u}dzd\bar{z}$ is the induced Riemannian metric on the Riemann surface and $\psi dz^3$ is a cubic differential defined on the surface. 

Using the fact that a Riemann surface of genus zero has no non-trivial holomorphic differentials, one can see that any minimal Lagrangian surface of genus zero in $\mathbb{C}P^2$ is totally geodesic, hence is the standard  immersion of $S^2$ in $\mathbb{C}P^2$ (\cite{Yau1974, Naitoh-Takeuchi82}). 
Any minimal Lagrangian immersed surface of genus one in $\mathbb{C}P^2$ can be  constructed in terms of algebraically completely integrable systems \cite{Sharipov, MM, McIn}. 


The objective of the present paper is to apply the loop group method \cite{DPW} to the construction of minimal Lagrangian surfaces in $\mathbb{C}P^2$.
The paper is organized as follows: in Section \ref{sec:Pre}, we recall the general theory  of loop group method for harmonic maps and basic set-ups for minimal Lagrangian surfaces in $\C P^2$. Then we give Wu's formula to describe the normalized potential of minimal Lagrangian surfaces. In Section \ref{sec:vacuum}, we characterize the vacuum solutions. Then in Section \ref{Sec:symmetries} we investigate minimal Lagrangian surfaces admitting some symmetries.  In Section \ref{sec:finiteorder}, we give a complete characterization of all full minimal Lagrangian surfaces with finite order symmetries with a fixed point or without fix points.

In all classes of integrable surfaces two types of one-parameter groups 
can be considered. At one hand these are the one-parameter groups of 
extrinsic isometries, i.e. one-parameter groups of isometries of the space the surfaces are contained in, and, on the other hand, one-parameter groups of isometries of the induced metric.
These two types of transformations were discussed first for surfaces of constant mean curvature in $\R^3$ by Smyth \cite{Smyth}.
For a given minimal Lagrangian surface $f: M \rightarrow \C P^2$ one considers, more generally now, as extrinsic isometries  one parameter groups $(\gamma_t, R_t)$ 
with $\gamma_t$ a one-parameter group of automorphisms of the Riemann surface $M$ and $R_t$ a one-parameter group of isometries of $\C P^2$ satisfying $f(\gamma_t.z) = R_t f(z)$. It was shown (see  \cite{DoMaNewLook}, \cite{DoMaExplicit}) that up to 
normalizations only two types of such one-parameter groups occur: 
translationally equivariant surfaces and rotationally equivariant surfaces.
In the case of minimal Lagrangian surfaces it follows that the domain of definition is all of $\C$ or $S^2$. The rotationally equivariant surfaces all yield minimal Lagrangian spheres and real projective spaces.
In  \cite{DoMaNewLook}, \cite{DoMaExplicit} also an explicit construction of all translationally equivariant minimal Lagrangian surfaces is given, extending the result of \cite{CU94}. 
Moreover, translationally equivariant cylinders have been discussed there and it was shown, how one can recover the results on tori of \cite{CU94} by the methods of the present paper.
For the second type of one-parameter groups, a coarse classification of  minimal Lagrangian surfaces with one-parameter groups of self-isometries of $M$ relative to the induced metric is given in Theorem \ref{Th6.5}.
No literature about this is known to the authors.

In Section \ref{Sec:entireradiallysym}, we  discuss radially symmetric,  minimal Lagrangian surfaces, as these surfaces are commonly known and restrict from thereon to the case of immersions defined on all of $\C$. 
In this paper we only discuss entire surfaces. As a matter of fact, in Section \ref{Sec:radially symmetric} we give the construction of entire radially symmetric minimal Lagrangian immersions into $\C P^2$ with constant normalized potential. This contains the simplest case and still seems to be new.
 

\section{Preliminaries}
\label{sec:Pre}

We first recall the loop group method for the construction of primitive harmonic maps of a Riemann surface into a $k$-symmetric space \cite{DPW}.
\subsection{Primitive harmonic maps}

As pointed out above, minimal Lagrangian spheres in $\C P^2$ are completely known. We will therefore assume without loss of generality that the simply-connected cover of any Riemann surface considered in this paper is non-compact, whence contractible and we will usually assume that the Riemann surface is the complex plane, the upper half-plane or the open unit disk. 
We will always write $\D$ for such a Riemann surface.

Let $G$ be a  compact real semisimple Lie group with an automorphism $\sigma: G\rightarrow G$ of finite order $k\geq 2$. Let $G^{\sigma}=\mathrm{Fix}(G,\sigma)$ and $G^{\sigma}_0$ the identity component of $G^{\sigma}$. A real homogeneous space $G/K$ is called a $k$-symmetric space, if
$(G)^{\sigma}_0\subset K\subset G^{\sigma}$.
Then $\sigma$ induces an automorphism, also denoted by $\sigma$, of the Lie algebra $\mathfrak{g}$  of $G$ as well as of their complexifications $G^\C$ and 
$\mathfrak{g}^\C.$ 

The eigenspace decomposition of $\mathfrak{g}^\C$ gives a reductive decomposition
of $\mathfrak{g}$ 
\begin{equation}\label{eq:gkm}
\mathfrak{g}=\mathfrak{k}\oplus \mathfrak{m}, 
\end{equation}
where $ \mathfrak{m}^\C =\sum_{l\in \mathbb{Z}_k\backslash\{0\}}\mathfrak{g}^\C_l$ and $\mathfrak{g}^\C_l$ is the $\epsilon^l$-eigenspace of  $\sigma$ for $\epsilon=e^{2\pi i/k}$.

Let  $f: \D \rightarrow G/K$ be a smooth map with frame (or lift) $\F:\D \rightarrow G$ so that $f=\pi\circ \F$, where $\pi: G\rightarrow G/K$ is the canonical projection.
Such a frame always exists globally.

Then the $\mathfrak{g}$-valued $1$-form $\alpha=\F^{-1}d\F$ has a split $\alpha=\alpha_{\mathfrak k}+\alpha_{\mathfrak m}$ according to \eqref{eq:gkm}. 
Write $\alpha_{\mathfrak m}=\alpha_{\mathfrak m}^{\prime}+\alpha_{\mathfrak m}^{\prime\prime}$ 
according to the type decomposition $TM^{\mathbb C}=T^{\prime}M\oplus T^{\prime\prime}M$ for $M = \D$. Then we have the well known result:
A map $f: \D\rightarrow G/K$ is harmonic if  and only if
\begin{eqnarray*}
&&d\alpha_{\mathfrak m}^{\prime}+[\alpha_{\mathfrak k}\wedge \alpha_{\mathfrak m}^{\prime}]=0,\\
&& d\alpha_{\mathfrak k}+\frac{1}{2} [\alpha_{\mathfrak k}\wedge \alpha_{\mathfrak k}]+[\alpha_{\mathfrak m}^{\prime}\wedge \alpha_{\mathfrak m}^{\prime\prime}]=0.
\end{eqnarray*}

A map $f: \D\rightarrow G/K$ is called \emph{primitive}  if $\alpha_{\mathfrak m}^{\prime}$ takes values in $\mathfrak{g}_{-1}$. 
Remark that when $k>2$, a primitive map is automatically harmonic 
\cite{BuP1994},  Theorem 3.6. (Also see \cite{Black}.) 

Combining these two cases, we say that $f$ is a primitive harmonic map
if either   $k=2$ and $f$ is harmonic, or  $k>2$ and $f$ is primitive.

Now if $F$ frames a primitive harmonic map, then 
$$\alpha_{\lambda}=\lambda^{-1}\alpha_{\mathfrak m}^{\prime}+\alpha_{\mathfrak k}+\lambda \alpha_{\mathfrak m}^{\prime\prime}$$  
is a family of flat connections for each $\lambda\in \mathbb{C}^{*}$.
Therefore, by  solving $\F_{\lambda}^{-1}d \F_{\lambda}=\alpha_{\lambda}$,
we obtain the family $\F_{\lambda}$ of \emph{extended frames}.
We will always interpret  the family  $\F_{\lambda}$
as a map from $\D$ into the twisted loop group
$$\Lambda G_{\sigma}=\{g: S^1\rightarrow G \text { smooth } | \, g(\epsilon \lambda)=\sigma (g(\lambda))\}.$$
Conversely, given an extended frame $\F_{\lambda}$, then $\F_{\lambda}$ frames a primitive harmonic map for each $\lambda\in S^1$. 

Using the point $0 \in \D$ as a base point (which is what we usually do, unless we state the opposite explicitly),  the space of primitive harmonic maps $f: \D\rightarrow G/K$ with $f(0)=eK$ can be identified with the space of extended frames 
$\F: \D\rightarrow \Lambda G_{\sigma}$ with $\F(0)=k\in K$ modulo gauge transformations $H: \D\rightarrow K$ via the relation $f=\pi\circ \F_{\lambda=1}$.  We will
usually implement this bijection by choosing without loss of generality $\F(0,\lambda) = e$.
Thus constructing primitive harmonic maps turns equivalently into the construction of an extend frame from a contractible Riemann surface into a twisted loop group.
The crucial property of the loop group method introduced in \cite{DPW} is that any primitive harmonic map defined on a contractible Riemann surface $\D$ can be obtained from some \emph{potential}, which is a $1$-form on $\D$, holomorphic in $\lambda \in \C^*$, and can be assumed to be holomorphic in $z \in \D$.  
In this case the potential may be an infinite power series in $\lambda$.
However, by weakening the assumption to \emph{meromorphic in $z$}
 one can assume that the potential only contains one power of $\lambda$,  namely $\lambda^{-1}$.

\subsection{Loop groups}
 
Let's introduce some notation. By $\mathbf{D}$ we
denote the interior of the unit disk $\mathbf{D}=\{\lambda\in \mathbb{C}| |\lambda|<1\}$ and by $\E$  the exterior of the unit disk, 
$\E=\{\lambda\in \mathbb{C}| |\lambda|>1\} \cup {\infty}$
in $S^2$.

Since we are primarily interested in groups and loop groups related to $SL(3, \C)$, we  write down the conditions below for simply-connected complex (matrix) Lie groups and compact stabilizer groups $K$ only.

Set 
\begin{eqnarray*}
\Lambda G_{\sigma}^{\mathbb C}&=&\{ g: S^1\rightarrow G^{\mathbb C} | g
\text{  has finite Wiener norm},\, g(\epsilon \lambda)=\sigma g(\lambda)\}, \\
\Lambda^{+} G^{\mathbb C}_{\sigma}&=&\{g\in \Lambda G^{\mathbb{C}}_{\sigma} | \, g \text{ extends holomorphically to } \mathbf{D}, g(0)\in K^{\mathbb C}\},\\
\Lambda^{+}_B G_{\sigma}^{\mathbb C}&=&\{ g\in \Lambda^{+} G_{\sigma}^{\mathbb C} |\, 
g(0)\in B \},\\
\Lambda^{-} G^{\mathbb C}_{\sigma}&=&\{ g\in \Lambda G^{\mathbb{C}}_{\sigma} | \, g \text{ extends holomorphically to } \mathbb{E}, g(\infty)\in K^{\mathbb C}\},\\
\Lambda^{-}_{*} G^{\mathbb C}_{\sigma}&=&\{ g\in \Lambda^{-} G^{\mathbb C}_{\sigma} | \, g(\infty)=e\},
\end{eqnarray*}
where $K^{\mathbb C} = KB$ is a fixed Iwasawa decomposition of $K^{\mathbb C}$.

We will always equip $\Lambda G_{\sigma}^{\mathbb C}$ with 
the Wiener topology of absolute convergence of the Fourier coefficients. Then the group
$\Lambda G^{\mathbb C}_{\sigma}$ becomes a complex Banach Lie group  with Lie algebra
\begin{equation*}
\Lambda \mathfrak{g}^{\mathbb{C}}_{\sigma}:=\{\xi: S^1 \rightarrow \mathfrak{g}^{\mathbb C}|
 \sigma(\xi(\lambda))=\xi(\epsilon \lambda)\}.
\end{equation*}

If $\xi \in \Lambda \mathfrak{g}^{\mathbb C}_{\sigma}$, its Fourier decomposition is
$$\xi=\sum_{l \in \mathbb{Z}} \lambda^l\xi_l, \quad \xi_l \in \mathfrak{g}_l$$
and the Lie subalgebras of $\Lambda \mathfrak{g}^{\mathbb C}_{\sigma}$ corresponding to the subgroups
$\Lambda G_{\sigma}$, $\Lambda^{+} G^{\mathbb C}_{\sigma}$ and $\Lambda^{-} G^{\mathbb C}_{\sigma}$ are
\begin{eqnarray*}
\Lambda \mathfrak{g}_{\sigma}&=&\Lambda \mathfrak{g}^{\mathbb C}_{\sigma}\cap \mathfrak{g},
\\
\Lambda^{+}\mathfrak{g}^{\mathbb C}_{\sigma}&=&\{\xi\in \Lambda \mathfrak{g}^{\mathbb C}_{\sigma} |
\xi_l=0 \text{ for } l<0, \xi_0\in \mathfrak{k}\},\\
\Lambda^{-}\mathfrak{g}^{\mathbb C}_{\sigma}&=&\{\xi\in \Lambda \mathfrak{g}^{\mathbb C}_{\sigma} |
\xi_l=0 \text{ for } l>0, \xi_0\in \mathfrak{k}\}.
\end{eqnarray*}

Similar conditions hold for the remaining two Lie algebras.

We finish this subsection by quoting the two splitting theorems which are of 
crucial importance for the application of the loop group method.

The first of these theorems is due to Birkhoff, who invented it for the loop group of $GL(n,\C)$ in an attempt to solve Hilbert's 21'{st} problem. 

\begin{theorem}[Birkhoff Decomposition]
Let $G$ be a compact real Lie group. Then the multiplication $\Lambda^{-}_{*} G^{\mathbb C}_{\sigma} \times \Lambda^{+} G^{\mathbb C}_{\sigma} \rightarrow \Lambda G^{\mathbb C}_{\sigma}$ is a 
complex analytic diffeomorphism onto the open, connected  and dense subset 
$\Lambda^{-}_{*} G^{\mathbb C}_{\sigma}\cdot \Lambda^{+} G^{\mathbb C}_{\sigma}$ of  $\Lambda G^{\mathbb C}_{\sigma},$
called the big (left Birkhoff) cell. 

In particular, if $g\in \Lambda G^{\mathbb C}_{\sigma}$ is contained in the big cell,  then
$g$ has a unique decomposition $g=g_{-}g_{+}$, where $g_{-}\in \Lambda_{*}^{-} G_{\sigma}^{\mathbb C}$ 
and $g_{+}\in \Lambda^{+}G^{\mathbb C}_{\sigma}$.
\end{theorem}

The second crucial loop group splitting theorem is the following

\begin{theorem}[Iwasawa decomposition]

Let $G$ be a real compact Lie group. Then the multiplication map 
$$\Lambda G_{\sigma}\times \Lambda^{+}_B G_{\sigma}^{\mathbb C} 
\rightarrow \Lambda G_{\sigma}^{\mathbb C}$$
is a real-analytic diffeomeorphism of Banach Lie groups.
\end{theorem}

This result is well known for untwisted loop groups (see, e.g. Pressley-Segal \cite{PS})  and was extended to the twisted setting in  \cite{DPW}.  

\subsection{The basic loop group method}
\label{subsection:loop method}

Let us start from some primitive harmonic map $f: \D \rightarrow G/K$ and let's consider an  extended frame  $\F: \D \rightarrow \Lambda G_{\sigma}$ of $f$. 
Unless stated otherwise we will always assume $\F(0,\lambda) = e$.

While the frame $\F$ satisfies a non-linear integrability condition, the objects we construct next trivially satisfy the integrability condition.
\vspace{2mm}

{\bf Construction 1: Holomorphic potentials}
\vspace{2mm}

For any extended frame $\F: \D\rightarrow \Lambda G_{\sigma}$ with $\F(0)=e$, of some primitive harmonic map $f$ one can show that there exists a global matrix function
$V_{+}: \D\rightarrow \Lambda^{+}G_{\sigma}^{\mathbb C}$ solving the $\bar{\partial}$-problem
\begin{equation}
 \bar\partial {V_{+}} V_{+}^{-1}=-(\alpha_{\mathfrak k}^{\prime\prime}+\lambda \alpha_{\mathfrak m}^{\prime \prime} ), \quad\quad
 V_{+}(0)=e
\end{equation}
 over   $\D$, so that $C=\F V_{+}$ gives a \emph{holomorphic 
extended frame}. The Maurer-Cartan form of $C$,  
$\eta=C^{-1}\partial C$  is a $(1,0)$-form defined on $\D$ and takes values in 
\begin{eqnarray*}
\Lambda_{-1,\infty}&:=&\{\xi\in \Lambda \mathfrak{g}_{\sigma}^{\mathbb C} | \xi \text{ extends holomorphically to } \\
&& \quad\quad 0<|\lambda|<1 
\text{ with a simple pole at } 0\}\\
&=&\{\xi=\sum_{l\geq -1} \lambda^l \xi_l\in \Lambda \mathfrak{g}_{\sigma}^{\mathbb C}\}. 
\end{eqnarray*}
This differential 1-form on $\D$ is called a \emph{holomorphic potential} for $f$.

Conversely, starting from a holomorphic $(1,0)$-form $\eta=\sum_{l\geq -1} \lambda^l \eta_l  \in \Lambda_{-1,\infty}$, one  first solves the ODE $dC=C\eta$, $C(0, \lambda) = e$ over  $\D$, where $C\in \Lambda G^{\mathbb C}_{\sigma}$. Performing an Iwasawa decomposition of $C$: $C={\mathbb F}V_{+}$,
where $\mathbb{F}=\mathbb{F}(z,\lambda) \in \Lambda G_{\sigma}$
and $V_{+}=V_0+\lambda V_1+\lambda^2 V_2+\cdots \in \Lambda^+ G^{\mathbb C}_{\sigma}$, it turns out that $\F$ is the extended frame of some primitive harmonic map $f=\pi\circ \F_{\lambda=1}:\D\rightarrow G/K.$
\vspace{2mm}

Altogether we obtain

\begin{theorem}[\cite{DPW}] Let $G/K$ be a compact $k$-symmetric space.
Let $f: \D\rightarrow G/K$ be a primitive harmonic map with $f(0)=eK$ and $\F: \D \rightarrow \Lambda G_{\sigma}$  an extended frame of $f$ satisfying $\F(0,\lambda) = e$. 
Then there exists a matrix function 
$V_+: \D \rightarrow  \Lambda^{+}G_{\sigma}^{\mathbb{C}}$ such that 
$C = \F V_+$ is holomorphic in $z \in \D$. Furthermore,  $\eta = C^{-1} dC \in \Lambda_{-1,\infty}$ is a holomorphic $(1,0)$-form on $\D$, called a \emph{holomorphic potential} for $f$.

 Conversely, given a  holomorphic $(1,0)$-form $\eta \in \Lambda_{-1,\infty}$ on $\D$ 
we obtain a map  $C:\D \rightarrow \Lambda G_{\sigma}^{\mathbb C}$, 
satisfying $dC = C \eta$ and $C(0,\lambda)=e$.
Performing an Iwasawa decomposition of $C$ we obtain an extended frame $\F: \D \rightarrow \Lambda G_{\sigma}$ of some  primitive harmonic map $f=\pi\circ \F_{\lambda=1}$. 
\end{theorem}

{\bf Construction 2: Normalized potentials}
\vspace{2mm}

If one does not require that $C = \F V_+$ is necessarily holomorphic in $z$, but only meromorphic,  then by using the  Birkhoff decomposition  
$\F_- = \F V_+$, \cite{DPW} shows that any  harmonic map 
$f: \D\rightarrow G/K$ can  be obtained from a meromorphic potential of the form $ \F_-^{-1} d\F_- = \mu=\lambda^{-1}\mu_{-1}$. 
Note, in this step $\F_-$ is automatically meromorphic on $\D$.

The converse procedure follows the pattern outlined above.
We collect the results for the meromorphic case by

\begin{theorem}[\cite{DPW}] Let $G/K$ be a compact $k$-symmetric space.
Let $f: \D\rightarrow G/K$ be a primitive harmonic map with $f(0)=eK$ and $\F: \D \rightarrow \Lambda G_{\sigma}$  an extended frame of $f$ satisfying $\F(0,\lambda) = e$. 
Then there exists a discrete subset $S\subset \D\backslash \{0\}$ such that for any point $z\in \D\backslash S$, the Birkhoff decomposition 
$\F(z,\cdot)=\F_{-}(z,\cdot) \F_{+}(z,\cdot)$ exists, with 
$\F_{-}(z,\cdot)\in \Lambda^{-}_{*}G_{\sigma}^{\mathbb{C}}$ and 
$\F_{+}(z,\cdot)\in \Lambda^{+}G_{\sigma}^{\mathbb{C}}$, and 
$\eta= \F_{-}(z,\lambda)^{-1}d\F_{-}(z,\lambda)$ is a $\mathfrak{m}^{\mathbb C}$-valued meromorphic $(1,0)$-form with poles in $S$ and which only contains the power $\lambda^{-1}$.

 Conversely, given a  $\mathfrak{m}^{\mathbb C}$-valued meromorphic $(1,0)$-form $\eta$ on $\D$ containing only the power $\lambda^{-1}$, for which the solution to $\F_{-}(z,\lambda)^{-1}d\F_{-}(z,\lambda)= \eta$ with $\F_{-}(0,\cdot)=e$ is meromorphic, we obtain a map  $\F_{-}:\D\backslash S \rightarrow \Lambda_{*}^{-} G_{\sigma}^{\mathbb C}$, where the discrete subset $S\subset \D\backslash\{0\}$ consists of the poles of $\eta$.
Performing an Iwasawa decomposition of $\F_-$ we obtain an extended frame $\F: \D\backslash S\rightarrow \Lambda G_{\sigma}$ of some  primitive harmonic map $f$  which satisfies $\F(0,\lambda) = e$. 

The two constructions explained in this theorem are inverse to each other. 
\end{theorem}

\begin{remark} \begin{enumerate}
\item The  $\mathfrak{m}^{\mathbb C}$-valued meromorphic $(1,0)$-form $\eta$  on $\D$,  unique after the choice of a base point, 
is called the {\it normalized potential of $f$ with the point $0$ as the reference point}.  
\item  We would like to emphasize  that $\bar\partial \F_{-}(z,\lambda)=0$ on $\D\backslash S$.  
 
\item  In the generality discussed above, a meromorphic $\eta$ will in general not yield a globally smooth minimal Lagrangian immersion.  To obtain global smoothness one needs to require certain relations between the poles and zeros of the coefficients of the potential. For CMC surfaces in $\R^3$, see e.g. \cite{DoHaMero}.
\end{enumerate}
\end{remark}

More information concerning the loop group method can be found in 
 \cite{DoHa;sym1}, \cite{DoHa;sym2}, \cite{DoWasym1}.
 
\subsection{Minimal Lagrangian surfaces in $\mathbb{C}P^2$ }
\label{subsec:miL}

We recall briefly the basic set-up for minimal Lagrangian surfaces in $\mathbb{C}P^2$. For details we refer to \cite{MM} and references therein.

Let $\mathbb CP^2$ be the complex projective plane endowed with the Fubini-Study metric of constant holomorphic sectional curvature $4$. For a minimal Lagrangian immersion of a nonorientable surface, some double cover is orientable. Hence it is no restriction to assume in this paper  that $f:M \rightarrow {\mathbb C}P^2$ is a minimal Lagrangian immersion of an oriented surface. 

The induced metric on $M$ generates a conformal structure with
respect to which the metric is $g=2e^{u} dzd{\bar z}$, and where $z=x+iy$ is a local
conformal coordinate on $M$ and $u$ is a real-valued function defined on $M$ locally.
 
For our approach in this paper we will need certain lifts to  $S^5(1) =\{Z\in \mathbb{C}^3 | Z\cdot \bar{Z}=1\}$, 
where $Z\cdot \overline{W}=\sum_{k=1}^{3} z_{k} \overline{w_{k}}$
denotes the Hermitian inner product for any $Z=(z_1,z_2,z_3)$ and $W=(w_1,w_2,w_3)\in \mathbb{C}^3$.

 \begin{theorem}
Let $f: M \rightarrow {\mathbb C}P^2$ be a Lagrangian immersion of an oriented surface and $U$ an open contractible subset of $M$. 
Then the Lagrangian immersion $f_{U}:=f|_{U}$ has a horizontal lift ${\f}_U: U \rightarrow S^5(1) $, i.e. ${\f}_U$ satisfies the  equations
\begin{equation}\label{horizontal}
({\f}_U)_z \cdot {\overline {{\f}_U}}=({\f}_U)_{\bar z} \cdot {\overline {{\f}_U}} =0.
\end{equation}
Moreover, ${\f}_U$ is uniquely determined by this property up to a constant factor $\delta \in S^1$. 
\end{theorem}
\begin{proof}
First we note, that $f$ has a lift $\check{\f}_U$ to $S^5(1) $, since the pullback of the Hopf  fibration $S^5(1) \rightarrow {\mathbb C}P^2$ is trivial over the contractible subset $U$.

Secondly, it is straightforward to verify that
this lift  $\check{\f}_U$  induces  for any Lagrangian immersion  the closed one-form 
$d \check{\f}_U\cdot \overline{\check{\f}_U}$.
Hence, since $U$ is contractible,  there exists a real function $\eta\in C^\infty (U)$ such that 
$i d\eta= d \check{\f}_U\cdot \overline{\check{\f}_U}$.
Then ${\f}_U =e^{-i\eta} \check{\f}_U$ is a  horizontal lift of $f$ from $U$ to $S^5(1)$.
If $\mathfrak{g}_U$ is another horizontal lift, then  $\mathfrak{g}_U= e^{ih}\f_U$ by the definition of the Hopf fibration. Then a straightforward computation using  (\ref{horizontal}) implies that the function $e^{ih}$ is a constant $\delta$ and again the definition of the Hopf fibration implies  $\delta \in S^1$.
\end{proof}

Remark that an immersed surface $\f: M \rightarrow S^5(1)$ is called \emph{Legendrian} with respect to the standard contact form if it satisfies (\ref{horizontal}).

It is useful to point out that any  minimal Lagrangian immersion $f:M\rightarrow {\mathbb C}P^2$ 
has a natural lift  $\tilde{f}: \tilde{M} \rightarrow {\mathbb C}P^2$ to the universal cover $\tilde{M}$ as a  minimal Lagrangian immersion.

\begin{corollary} \label{uniquehori}
Let $f:M\rightarrow {\mathbb C}P^2$ be a minimal Lagrangian immersion of an oriented
surface and $\tilde{f}:\tilde{M} \rightarrow {\mathbb C}P^2$  its natural lift to the universal cover $\tilde{M}$ of $M$. Then, if  $M \neq S^2$, then $\tilde{M} $ is contractible and 
$\tilde{f}:\tilde{M} \rightarrow {\mathbb C}P^2$  can be lifted to a horizontal map
$\tilde{F}:\tilde{M} \rightarrow S^5(1)$ and this map is uniquely determined up to some constant factor $\delta \in S^1$.
\end{corollary}

 For general Riemann surfaces $M$  and general conformal immersions 
 $f: M \rightarrow \C P^2$ such a (global) lift  $\f:M \rightarrow S^5$ does not exist.

\begin{example}\label{Clifford} The map $f:\R^2\rightarrow S^5(1)$ defined by 
$$f(x,y)= \frac{1}{\sqrt{3}}(e^{2iy}, e^{i(\sqrt{3}x-y)}, e^{-i(\sqrt{3}x+y)})$$
is a flat, conformal, minimal Legendrian immersion in $S^{5}(1)$. It induces an embedded torus in $S^{5}(1)$
$$g: \R^2 /\Lambda \rightarrow S^5(1),$$
where $\Lambda=\Z \omega_1\oplus \Z \omega_2$ with $\omega_1=(0,2\pi)$ and $ \omega_2=(\frac{\pi}{\sqrt{3}}, \pi)$. Project this map by the Hopf fibration ${\varPi}_H: S^5(1)\rightarrow \C P^2$, we obtain a minimal Lagrangian torus 
$$f_0={\varPi}_H\circ g: \R^2 /\Lambda \rightarrow \C P^2.$$
 Notice that 
\begin{equation}\label{lamda_prime}
f((x,y)+\frac{k}{3}\omega_1+m\omega_2)=e^{-i\frac{2\pi}{3}k} f(x,y),
\end{equation}
where $k, m\in \Z$. Hence it induces a map $f_0^{\prime}: \R^2/ \Lambda_{C} \rightarrow \C P^2$, satisfying $f'_0 \circ {\varPi}' ={\varPi}_{H} \circ f$,  where $\Lambda_{C}=\Z\frac{\omega_1}{3}\oplus \Z\omega_2$ and ${\varPi}'$ is the natural projection 
from $\R^2$  to $\R^2/ \Lambda_{C}$.

We observe that $f_0^{\prime}$ is an embedded minimal Lagrangian  torus in 
$\C P^2$. This torus  is called the \emph{Clifford torus}. 

We claim that $f_0^{\prime}$ is not globally horizontally liftable.  Otherwise, assume that there exists a horizontal map
$$g^{\prime}: \R^2/\Lambda_{C} \rightarrow S^5(1)$$
such that ${\varPi}_H\circ g^{\prime}=f_0^{\prime}$. Then we have 
$${\varPi}_H\circ g^{\prime}\circ {\varPi}^{\prime}=f_0^{\prime}\circ {\varPi}^{\prime}={\varPi}_H\circ f.$$
Thus there exists $h: \R^2\rightarrow \R$ such that $g^{\prime}\circ {\varPi}^{\prime}=e^{ih(x,y)} f$.
This means
$$g^{\prime}\circ \tau\circ {\varPi}=e^{ih(x,y)} g\circ {\varPi},$$
where we set ${\varPi}: \R^2 \rightarrow\R^2/\Lambda$ and  $\tau: \R^2/\Lambda \rightarrow \R^2/\Lambda_{C}$ are the natural projections, whence 
${\varPi}' = \tau \circ {\varPi}$.
Since $g$ and $g^{\prime} \circ \tau$ are horizontal, $h$ is constant and hence we have 
$$g^{\prime}\circ {\varPi}^{\prime}=e^{ih} f.$$
But the left hand side is invariant under the lattice $\Lambda_{C}$ and the right hand side depends on $k$ due to \eqref{lamda_prime}. This contradiction shows the claim holds.
\end{example}

We would like to emphasize again, that in this paper all Lagrangian surfaces are assumed to be defined on a contractible domain. Hence we exclude $M =S^2$
throughout this paper  (except where the opposite is stated explicitly) and will 
always use $M = \D$ as domains of minimal Lagrangian surfaces, indicating that $\D$ is a contractible domain, usually assumed to be $\C$ or the upper-half plane or the unit disk.

\subsection{Frames for horizontal lifts}\label{subsec:frames}
In this subsection we assume that $\D$ is a contractible Riemann surface,  $f: \D \rightarrow \C P^2$ a minimal Lagrangian immersion and $\f: \D \rightarrow S^5(1)$ a horizontal lift of $f$.

The fact that the induced metric $g$ is conformal is equivalent to
\begin{equation}\label{fconf1} 
\begin{split}
&{\f}_z\cdot \overline{\f_z}=\f_{\bar z}\cdot \overline{\f_{\bar z}}=e^{u},\\
& \f_z\cdot \overline{\f_{\bar z}}=0.
\end{split}
\end{equation}
Thus 
$\mathcal{F}=(e^{-\frac{u}{2}}\f_z, e^{-\frac{u}{2}}\f_{\bar z}, \f)$ defines a Hermitian orthonormal moving frame on the surface $\D$.

It follows from \eqref{horizontal}, \eqref{fconf1} and the minimality of
 $f$ that $\mathcal{F}$ satisfies the frame equations (see e.g. \cite{MM})
\begin{equation}\label{eq:frame1}
\mathcal{F}_{z}=\mathcal{F} {\mathcal U}, \quad \mathcal{F}_{\bar z}=\mathcal{F} {\mathcal V},
\end{equation}
where
\begin{equation}\label{eq:UV1}
{\mathcal U}=\left(\begin{array}{ccc}
                   \frac{u_z}{2} & 0 & e^{\frac{u}{2}} \\
      e^{-u}\psi  &-\frac{u_z}{2}  & 0 \\
       0 & -e^{\frac{u}{2}}  &0 \\
     \end{array}
   \right),
   \quad
{\mathcal V}=\left(
     \begin{array}{ccc}
      -\frac{u_{\bar z}}{2}  &  - e^{-u}\bar\psi  &0 \\
       0& \frac{u_{\bar z}}{2} &e^{\frac{u}{2}} \\
       -e^{\frac{u}{2}}  &0 &0 \\
     \end{array}
   \right),
\end{equation}
with
\begin{equation}\label{eq:phipsi}
\psi=\f_{zz}\cdot\overline{\f_{\bar z}}.
\end{equation}

Using (\ref{fconf1}) and Corollary \ref{uniquehori} one can easily check that the cubic differential $\Psi=\psi dz^3$ is actually independent of the choice of a lift and the complex coordinate $z$ of $\D$ and thus is globally defined on the Riemann surface $\D$. The differential $\Psi$ is called the \emph{Hopf differential} of $f$.

The compatibility condition of  the equations \eqref{eq:frame1} is
${\mathcal U}_{\bar z}-{\mathcal V}_z =[{\mathcal U},{\mathcal V}]$,  
and using \eqref{eq:UV1} this turns out to be equivalent to
 \begin{eqnarray}
u_{z\bar z}+e^u-e^{-2u}|\psi|^2&=0,\label{eq:mLsurfaces}\\
\psi_{\bar z}&=0.\label{eq:Codazzi}
\end{eqnarray}


Since  $\f$  it is uniquely determined up to a constant factor $\delta \in S^1$,
also $\mathcal{F}$ is only defined up to this  constant factor. 
We can actually assume 

\begin{proposition}
Let $\D$ be a contractible Riemann surface  and  $f: \D \rightarrow {\mathbb C}P^2$ a minimal Lagrangian immersion. Then for the corresponding  frame $\mathcal{F}$ we can assume without loss of generality
$\det \mathcal{F} = -1$.
Under this assumption $\mathcal{F}$ is uniquely determined up to some factor $\delta$ satisfying $\delta^3 = 1$.
\end{proposition}
\begin{proof}
We know from orthonormality that $\det \mathcal{F} = \kappa \in S^1$.
 Moreover, by \eqref{eq:frame1}-\eqref{eq:UV1}, we see that
$d \ln \det \mathcal{F} =  0$, which implies that the determinant of $\mathcal{F}$ is independent of $z$. Hence we can multiply 
$\mathcal{F}$ by a  constant in $S^1$ such  that $\det \mathcal{F} = -1$. 
\end{proof}

From here on we will always assume that $ \det \mathcal{F} = -1$ holds.

Notice that the integrability conditions  \eqref{eq:mLsurfaces}-\eqref{eq:Codazzi} are invariant under the transformation
$\psi \rightarrow \nu \psi$ for any $\nu \in S^1$.
This implies that after replacing $\psi$ in \eqref{eq:UV1}   by  $\psi^\nu=\nu\psi$ the equations \eqref{eq:frame1} are still integrable.
Therefore, the solution $\mathcal{F}(z,\bar{z}, \nu)$, with initial condition
$\mathcal{F}(0,0, \nu) = \mathcal{F}(0,0)$
 to this new system of equations is the frame of some minimal Lagrangian surface $f^\nu$.
The argument above  yields also now that $\det \mathcal{F} (z, \bar z, \nu) = \delta(\nu) \in S^1$ is independent of $z$.

Next we consider $A(z, \bar z, \nu) = \partial_\nu \mathcal{F} (z, \bar z, \nu) \cdot  \mathcal{F} (z, \bar z, \nu) ^{-1}$. It is straightforward to compute
$d A(z, \bar z, \nu) = \mathcal{F} (z, \bar z, \nu) \partial_\nu \beta \mathcal{F} (z, \bar z, \nu)^{-1}$, 
where $\beta = \mathcal{F} (z, \bar z, \nu)^{-1} d \mathcal{F} (z, \bar z, \nu) $.

Therefore we obtain $0 = \tr d A(z, \bar z, \nu) = 
d \tr  (\partial_\nu \mathcal{F} (z, \bar z, \nu) \mathcal{F} (z, \bar z, \nu)^{-1})$,
whence $\partial_\nu \ln \det \mathcal{F} (z, \bar z, \nu)  = 
\tr (\partial_\nu \mathcal{F} (z, \bar z, \nu)  \mathcal{F} (z, \bar z, \nu)^{-1} )= h(\nu)$ is independent of $z$. Setting $z = 0$ and recalling our choice of initial condition we obtain $h(\nu) \equiv 0$. Thus $\det \mathcal{F}(z,\bar{z}, \nu)$ is independent of $z$ and of $\nu$.
As a consequence we can again assume without loss of generality   $\det \mathcal{F}(z,\bar{z}, \nu) = -1$. 

It turns out to be convenient to consider in place of the frames $\mathcal{F}(z,\bar{z}, \nu)$ the gauged frames 
\begin{equation} \label{F}
\mathbb{F}(\lambda)=\mathcal{F}(\nu)\left(
                                   \begin{array}{ccc}
                                     -i\lambda & 0 & 0 \\
                                     0 &  -i\lambda^{-1} & 0 \\
                                     0 & 0 &  1 \\
                                   \end{array}
                                 \right),
\end{equation}
where $i\lambda^3 \nu=1$. Note that we have  $\det \mathcal{F} = -1$ and  
$\det \F = 1$, 
hence  $\F \in \Lambda SU(3)$.
 
By replacing $\mathbb{F}(z,\bar z, \lambda)$ by 
$\mathbb{F}(0,0, \lambda)^{-1} \mathbb{F}(z,\bar{z}, \lambda)$ we can assume without loss of generality  $\mathbb{F}(0,0, \lambda) = I$.
  Note that this normalization implies $\f(0,0, \lambda) = e_3$ and also implies 
  that the minimal Lagrangian immersion   $f(z, \bar z,\lambda) $ satisfies
 $f(0,0, \lambda) = [e_3]$. Thus the change above is the one moving the associated family $f(z, \bar z, \lambda) $ of the immersion $f$ to the associated family 
  of the immersion 
 $\mathbb{F}(0,0)^{-1} f(z,\bar{z}).$ 
 Moreover, the immersions for which the above normalizations hold are uniquely determined and, in particular, independent of the horizontal lift.
 
For these frames we obtain the equations
\begin{equation}\label{eq:mathbbF}
\begin{split}
\mathbb{F}^{-1}\mathbb{F}_z&=
\frac{1}{\lambda}\left(
   \begin{array}{ccc}
     0 & 0 & i e^{\frac{u}{2}} \\
  -i \psi e^{-u}   & 0 & 0 \\
     0 &  i e^{\frac{u}{2}} & 0 \\
   \end{array}
 \right)+\left(
   \begin{array}{ccc}
   \frac{u_z}{2}  & &  \\
      & -\frac{u_z}{2} &  \\
      & & 0 \\
   \end{array}
 \right)\\
 &:=\lambda^{-1}U_{-1}+U_0,\\
\mathbb{F}^{-1}\mathbb{F}_{\bar{z}} &=\lambda \left(
   \begin{array}{ccc}
     0 &  -i\bar{\psi} e^{-u} & 0 \\
     0& 0   & i e^{\frac{u}{2}} \\
     i e^{\frac{u}{2}} &0 & 0 \\
   \end{array}
 \right)+\left(
   \begin{array}{ccc}
     -\frac{u_{\bar z}}{2} & &  \\
      & \frac{u_{\bar z}}{2}  & \\
     & & 0 \\
   \end{array}
 \right)\\
  &:=\lambda V_{1}+V_0.
\end{split}
\end{equation}

\begin{proposition}\label{prop:frame}
 Let $M$ be an arbitrary  Riemann surface different from $S^2$ and $U$ a contractible open subset of $M$.
Let $\mathbb{F}(z, \bar{z},\lambda),  \lambda\in S^1, z \in U, $ be a solution to the system \eqref{eq:mathbbF}.
Then $[\mathbb{F}(z, \bar{z}, \lambda)e_3]$ gives a minimal Lagrangian surface defined on $U$ with values in $\mathbb{C}P^2$ and 
with the metric $g=2e^{u}dzd\bar{z}$ and the Hopf differential $\Psi^{\nu}=\nu \psi dz^3$ with $i\lambda^3 \nu=1$.

Conversely, suppose $f^{\nu}:M\rightarrow \mathbb{C}P^2$ is a conformal parametrization of
a minimal Lagrangian surface in $\mathbb{C}P^2$ with the metric $g=2e^{u}dzd\bar{z}$ and Hopf differential
$\Psi^{\nu}=\nu\psi dz^3$. Then  for any open, contractible subset $U$ of $M$ there exists  a  frame $\mathbb{F}: U \rightarrow SU(3)$
satisfying \eqref{eq:mathbbF} with $i\lambda^3 \nu=1$. This frame is unique if we choose a base point $z_0$ and normalize  $\mathbb{F}(z_0, \bar{z}_0,\lambda) = I$.
\end{proposition}

We have already pointed out that one can interpret the frames above as elements of the loop group $\Lambda SU(3)$. It actually turns out that they belong to a smaller, twisted, loop group.

\subsection{The loop group method for minimal Lagrangian surfaces}

 Let $\sigma$ denote the  automorphism of $G^{\mathbb{C}}=SL(3,\mathbb{C})$ of order $6$ defined by
\begin{equation*}
\sigma: g\mapsto P (g^t)^{-1} P^{-1}, \quad \text{ where } 
P=\left(
  \begin{array}{ccc}
    0 & \epsilon^2 & 0 \\
    \epsilon^4 & 0 & 0 \\
    0 &0 & 1\\
  \end{array}
\right), \text{ with } \epsilon=e^{\pi i/3}.
\end{equation*}
Let $\tau$ denote the anti-holomorphic involution of $G^{\mathbb{C}}=SL(3,\mathbb{C})$ which defines  the real form $G=SU(3)$, 
given by $$\tau(g):=(\bar{g}^t)^{-1}.$$
Then on the Lie algebra level the corresponding automorphism $\sigma$ of order $6$ and 
the anti-holomorphic automorphism $\tau$
of $\mathfrak{g}^{\mathbb{C}}=sl(3,\mathbb{C})$ are
\begin{equation*}
\sigma: \xi \mapsto -P\xi^t P^{-1}, \quad \tau: \xi \mapsto -\bar{\xi}^t.
\end{equation*}

Explicitly the eigenspaces $\mathfrak{g}_k^\C$ of $\sigma$ with respect to the eigenvalue $\epsilon^k$ in $sl(3,\mathbb{C})$ are given as follows
\begin{equation*}
\begin{split}
\mathfrak{g}_0^\C&=\left\{
                    \begin{pmatrix}
                    a &  &  \\
                     & -a &  \\
                     &  & 0 \\
                \end{pmatrix}
                \mid a\in \C
\right\},
\quad
\mathfrak{g}_1^\C=\left\{\begin{pmatrix}
                    0 & b & 0 \\
                     0 & 0 & a \\
                     a & 0 & 0 \\
                  \end{pmatrix}\mid a,b\in\C
\right\},
\\
\mathfrak{g}_2^\C&=\left\{
                  \begin{pmatrix}
                    0 & 0& a  \\
                     0& 0 & 0 \\
                     0& -a & 0 \\
                  \end{pmatrix} \mid a\in \C
              \right\},
\quad
\mathfrak{g}_3^\C=\left\{
                  \begin{pmatrix}
                    a &  &  \\
                      & a &  \\
                      &  & -2a \\
                  \end{pmatrix} \mid a\in \C
              \right\},\\
\mathfrak{g}_4^\C&=\left\{
                  \begin{pmatrix}
                    0 & 0 & 0  \\
                     0& 0 & a \\
                     -a & 0 & 0 \\
                  \end{pmatrix} \mid a\in \C
            \right\},
\quad
\mathfrak{g}_5^\C=\left\{
                  \begin{pmatrix}
                    0 & 0 & a \\
                     b & 0 & 0 \\
                     0 & a & 0 \\
                  \end{pmatrix} \mid a, b\in \C 
                \right\}.
\end{split}
\end{equation*}

In view of (\ref{eq:mathbbF}) and the eigenspaces stated just above one is led to consider the twisting automorphism
\begin{equation}
(\sigma (g))(\lambda) = \sigma ( g( \epsilon^{-1} \lambda)).
\end{equation}

Then the $\sigma$-twisted loop group is defined as in the beginning of this section
as fixed point set of this twisting automorphism. Now it is easy to verify

\begin{proposition}
The frames $\mathbb{F} (z, \bar{z}, \lambda)$ satisfying 
$\mathbb{F} (z_0, \bar{z}_0, \lambda) = I$ are elements of the twisted loop group
$\Lambda SU(3)_\sigma$. 
\end{proposition}

Using loop group terminology, we can state (refer to \cite{MM}):

\begin{proposition}
Let $f: \D \rightarrow \mathbb{C}P^2$ be a conformal parametrization 
of a contractible Riemann surface and $\f$ its horizontal lift.
Then the following statements are equivalent:
\begin{enumerate}
\item $f$ is minimal Lagrangian.
\item There exists a frame $\mathbb{F}: \D \rightarrow SU(3)$ which is primitive harmonic relative to $\sigma$. 
\item  There exists an extend frame $\F: \D\rightarrow \Lambda SU(3)_{\sigma}$ such that $\mathbb{F}^{-1}d\mathbb{F} =(\lambda^{-1} U_{-1}+U_0)dz+(\lambda V_1+V_0)d\bar{z} 
\subset \Lambda su(3)_{\sigma}$ is a one-parameter family of flat connections for all $\lambda \in \C^*$.
\end{enumerate}
\end{proposition}

 The general Iwasawa decomposition theorem  stated above takes in our case, i.e. for the groups $\Lambda SL(3,\mathbb{C})_{\sigma}$  and $\Lambda SU(3)_{\sigma}$, the following explicit form:

\begin{theorem}[Iwasawa decomposition theorem of $\Lambda SL(3,\mathbb{C})_{\sigma}$]\label{Thm:Iwasawa}
Multiplication $\Lambda SU(3)_{\sigma}\times \Lambda^{+} SL(3, \C)_{\sigma} \rightarrow \Lambda SL(3, \C)_{\sigma}$ is a diffeomorphism onto. Explicitly, 
every element $g\in \Lambda SL(3,\mathbb{C})_{\sigma}$ can be represented in the form
$g=h V_{+}$ with $h\in \Lambda SU(3)_{\sigma}$ and $V_{+}\in \Lambda^{+} SL(3,\mathbb{C})_{\sigma}$.
One can assume without loss of generality that $V_{+} (\lambda=0)$ has only positive diagonal entries.
In this case the decomposition is unique.
 \end{theorem}

\subsection{Wu's formula for minimal Lagrangian surfaces in $\mathbb{C}P^2$} \label{Wu's formula}

 The relation between normalized potentials and minimal Lagrangian immersions is a priori  quite abstract. It turns out, however, that there is a simple relation.

Let $f:\D \rightarrow \C P^2$ be a minimal Lagrangian immersion 
and $\mathbb{F}:\mathbb{D}\rightarrow \Lambda SU(3)_{\sigma}$ an extended frame for $f$,
which means that $\mathbb{F}^{-1}d\mathbb{F}=(\lambda^{-1}U_{-1}+U_0)dz+(\lambda V_1+V_0)d\bar{z}
\in  \Lambda \mathfrak{su}(3)_{\sigma}$.
In Subsection \ref{subsection:loop method} we have explained the basic construction scheme of normalized potentials.
Thus we perform a Birkhoff decomposition and obtain
$$\mathbb{F}=\mathbb{F}_{-}\mathbb{F}_{+}$$
for $z\in \mathbb{D}\backslash S$ and $S\subset \mathbb{D}$ is a discrete subset,
where
\begin{eqnarray}
\mathbb{F}_{+}&=&\mathbb{F}_0(I+\lambda \mathbb{F}_1+\cdots )\in \Lambda^{+}SL(3,\mathbb{C})_{\sigma},\\
\mathbb{F}_{-}&=&I+\lambda^{-1}B+\cdots \in \Lambda^{-}_{*}SL(3,\mathbb{C})_{\sigma}. \label{eq:F-}
\end{eqnarray}
Then $\mathbb{F}_{-}=\mathbb{F}\mathbb{F}_{+}^{-1}$ and
\begin{equation*}
\begin{split}
\eta&:=\mathbb{F}_{-}^{-1}d\mathbb{F}_{-}=\mathbb{F}_{+}(\mathbb{F}^{-1}d\mathbb{F})\mathbb{F}_{+}^{-1}-d\mathbb{F}_{+}\mathbb{F}_{+}^{-1}\\
&=\mathbb{F}_{+}((\lambda^{-1}U_{-1}+U_0)dz+(\lambda V_1+V_0)d\bar{z})\mathbb{F}_{+}^{-1}-d\mathbb{F}_{+}\mathbb{F}_{+}^{-1}\\
&=\lambda^{-1}\mathbb{F}_0U_{-1}\mathbb{F}_0^{-1}dz=:\lambda^{-1} A_{-1}dz.
\end{split}
\end{equation*}
Since $\eta$ is integrable, it follows that the matrix 
$$A_{-1}:=\mathbb{F}_0U_{-1}\mathbb{F}_0^{-1}$$
only depends on $z$, and actually is meromorphic in $z$. 

As usual, we assume $\mathbb{F}(0,0,\lambda)=I$ at the base point $z=\bar{z}=0$. 
Then we can expand the real analytic matrix functions $\F$, $\F_-$, $\F_+$ into power series in $z$ and $\bar{z}$. In these expansions we can set $\bar{z} =0$.
Then  from $\mathbb{F}=\mathbb{F}_{-}\mathbb{F}_{+}$ we obtain
$$\mathbb{F}(z,0,\lambda)=\mathbb{F}_{-}(z,0,\lambda)\mathbb{F}_{+}(z,0,\lambda).$$
Since
$$\mathbb{F}(z,0,\lambda)^{-1}\frac{d}{dz}\mathbb{F}(z,0,\lambda)=\lambda^{-1}U_{-1}(z,0)+U_0(z,0),$$
we infer
$\mathbb{F}(z,0,\lambda)\in \Lambda^{-}SL(3,\mathbb{C})_{\sigma}$.

Remembering that $\mathbb{F}_{-}(z,0,\lambda)$ is given by \eqref{eq:F-},
we conclude that
$\mathbb{F}_{+}(z,0,\lambda)=\mathbb{F}_0(z, 0)$ holds. In particular, it turns out that this matrix is independent of $\lambda$, i.e.,
$$\mathbb{F}(z,0,\lambda)=\mathbb{F}_{-}(z,0, \lambda)\mathbb{F}_0(z,0).$$

Forming the Maurer-Cartan forms on both sides we obtain
\begin{eqnarray*}
\lambda^{-1}U_{-1}(z,0)+U_0(z,0)&=&\mathbb{F}(z,0,\lambda)^{-1}\frac{d}{dz}\mathbb{F}(z,0,\lambda)\\
&=&\mathbb{F}_0^{-1}(\mathbb{F}_{-}^{-1}\frac{d}{dz}\mathbb{F}_{-})\mathbb{F}_0+\mathbb{F}_0^{-1}\frac{d\mathbb{F}_0}{dz}\\
&=&\lambda^{-1}\mathbb{F}_0^{-1}A_{-1}\mathbb{F}_0 +\mathbb{F}_0^{-1}\frac{d\mathbb{F}_0}{dz}.
\end{eqnarray*}
Thus we have
\begin{eqnarray*}
A_{-1}&=&\mathbb{F}_0 U_{-1}(z,0)\mathbb{F}_0^{-1},\\
\mathbb{F}_0^{-1}\frac{d\mathbb{F}_0}{dz}&=&U_0(z,0).
\end{eqnarray*}
In view of  $\mathbb{F}_0(0,0)=I$ and the specific forms of $U_0$ and $U_{-1}$ as stated in \eqref{eq:mathbbF},  
we obtain
$$\mathbb{F}_0(z,0)=\begin{pmatrix}
e^{\frac{u}{2}(z,0)-\frac{u}{2}(0,0)}& &\\
& e^{-\frac{u}{2}(z,0)+\frac{u}{2}(0,0)}&\\
& & 1
\end{pmatrix}$$
and finally we also obtain the formula for the normalized potential of $f$:
\begin{align} 
\eta&=\lambda^{-1}A_{-1}dz\nonumber\\
&=\lambda^{-1}\begin{pmatrix}
0&0&ie^{u(z,0)-\frac{u}{2}(0,0)}\\
-i\psi(z)e^{-2u(z,0)+u(0,0)}&0&0\\
0&ie^{u(z,0)-\frac{u}{2}(0,0)}&0
\end{pmatrix}
dz. \label{eq:Wu-eta}
\end{align}
Note that $\psi$ is holomorphic here.

\begin{proposition} [Wu's formula (\cite{Wu})]\label{prop:Wu's formula}
Let $f:\D \rightarrow \mathbb{C}P^2$ be a minimal Lagrangian immersion. Then with the notation of 
\eqref{eq:mathbbF} the normalized potential $\eta$ of $f$ with respect to the base point $z=0$ is given by the formula \eqref{eq:Wu-eta}.
\end{proposition}

\begin{remark}
(1) Wu's formula shows how the entries of the normalized potential can be expressed in terms of $u$ and $\psi$.

(2) The proof above gives an argument for small $z$. However, since $\eta$ is meromorphic on $\D$, all matrix entries have meromorphic extensions to $\D$. 

\end{remark}

\begin{example} For the Clifford torus $f: T^2_C=\C/\Lambda_C \rightarrow \mathbb{C}P^2$ (see details in Subsection \ref{subsec:miL}), one has the following commutative diagram 
\[
\begin{tikzcd}[column sep=4em,row sep=4em]
\C  \ar{r}{\f}  \ar{dr}{\tilde{f}} 
    \ar{d}[swap]{ \varPi}   &   S^5(1)   \ar{d}{\varPi_H}   \\
T^2_{C}   \ar{r}{f}  &   \C P^2
\end{tikzcd}
\]
and a horizontal lift $\f:\mathbb{C}\rightarrow S^5(1)$ given by
$$\f(z,\bar{z})=\frac{1}{\sqrt{3}} (e^{z-\bar{z}}, e^{\alpha z-\alpha^2 \bar{z}}, e^{\alpha^2 z-\alpha \bar{z}}),$$
where $\alpha=e^{\frac{2}{3}\pi i}$.
It is easy to see that $\psi=\f_{zz}\cdot\overline{\f_{\bar z}}=-1$ and $e^u=1$.
Then it follows from Wu's formula that the normalized potential of the Clifford torus is given by
\begin{equation*}
\eta=\lambda^{-1}\begin{pmatrix}
0&0&i\\
i&0&0\\
0&i&0
\end{pmatrix} dz.
\end{equation*}

We write $\eta = \lambda^{-1} A dz$ and verify 
$[A,\tau (A)]=0$. Therefore the solution to $d\F_{-} = \F_{-} $, $\F_{-}(0,\lambda) = I$ is given by $\F_{-}(z,\lambda) = \exp (z \lambda^{-1} A)$.

Performing the Iwasawa decomposition, we obtain the extended frame given by
$\F(z,\lambda) = \exp( z \lambda^{-1} A + \bar{z} \lambda \tau(A))$. 
Consider the translation 
$$z\mapsto z+\delta, \quad\quad \text{with } \delta\in \C.$$
Therefore the monodromy matrix of the frame $\F(z,\lambda)$ for this translation is given by
$$\F(z+\delta,\lambda)=M(\delta, \lambda)\F(z,\lambda),$$
where $$M(\delta,\lambda)=\exp(\delta\lambda^{-1}A+\bar{\delta} \lambda \tau(A)).$$
Then the map $f_{\lambda_0}: \C\rightarrow \C P^2$ can be defined on $\C/\delta \Z$ if and only if $f_{\lambda_0}(z+\delta)=f_{\lambda_0}(z)$,
which is equivalent to $M(\delta,\lambda_0)=cI$, where $c$ is a scalar  satisfying $c^3=1$.
Since the eigenvalues of $A$ are $i$, $i\alpha$ and $i\alpha^2$, it follows that the closing conditions for $\lambda_0\in S^1$ are
$$e^{i\lambda_0^{-1}\delta+i\lambda_0 \bar{\delta}}=e^{i\lambda_0^{-1}\alpha\delta+i\lambda_0 \alpha^2\bar{\delta}}
=e^{i\lambda_0^{-1}\alpha^2 \delta+i\lambda_0 \alpha \bar{\delta}}=c,$$
which is
\begin{eqnarray}
\mathrm{Re}(\lambda_0^{-1}\delta)&=&\frac{k\pi}{3}+l_1\pi, \label{periodcon1}\\
\mathrm{Re}(\lambda_0^{-1}\alpha\delta)&=&\frac{k\pi}{3}+l_2\pi, \label{periodcon2}\\
\mathrm{Re}(\lambda_0^{-1}\alpha^2\delta)&=&\frac{k\pi}{3}+l_3\pi, \label{periodcon3}
\end{eqnarray}
for $k=0,1$ or $2$ and $l_1,l_2,l_3\in \Z$. It is easy to see now that for any $\lambda_0 \in S^1$, 
the solutions to \eqref{periodcon1}-\eqref{periodcon3} are given by
$$\delta=\frac{2l_1-l_2-l_3}{3}\lambda_0 \pi +i\frac{l_3-l_2}{\sqrt{3}}\lambda_0 \pi.$$
Therefore, for arbitrary $\lambda_0$, we obtain the $\lambda_0=1$ lattice rotated by $\lambda_0$.
This implies the following
\begin{proposition} Every member in the associated family of the Clifford torus is still  a torus.
\end{proposition}

\end{example}

\section{Vacuum solutions}
\label{sec:vacuum}

A \emph{vacuum} is an extended frame whose normalized potential is given by $\eta=\lambda^{-1}A dz$ with $A\in \mathfrak{g}_{-1}$ a constant matrix satisfying  $[A,\tau(A)]=0$ (see \cite{BuP}). To clarify what this means we consider the  constant matrix
$$A=\begin{pmatrix}
0&0&a\\
b&0&0\\
0&a&0
\end{pmatrix}\in \mathfrak{g}_{-1}. 
\quad \text{ Then }
\tau(A)=\begin{pmatrix}
0&-\bar{b}&0\\
0&0&-\bar{a}\\
-\bar{a}&0&0
\end{pmatrix},$$
and the condition $[A,\tau(A)]=0$ says $|a|^2=|b|^2$.

Let's next write  $a=i re^{i\theta}$ and  $b=ir e^{i\beta}$.
Now take the following gauge transformation
$$\begin{pmatrix}
e^{i\delta}&&\\
&e^{-i\delta}&\\
&&1
\end{pmatrix}A
\begin{pmatrix}
e^{-i\delta}&&\\
&e^{i\delta}&\\
&&1
\end{pmatrix}=\begin{pmatrix}
0&0&ire^{i(\theta+\delta)}\\
ire^{i(\beta-2\delta)}&0&0\\
0&ire^{i(\theta+\delta)}&0
\end{pmatrix}.$$
Then choose $\delta$ such that $\theta+\delta=\beta-2\delta$,
i.e., $\delta=\frac{\beta-\theta}{3}$.
Thus, $$\eta=\lambda^{-1}ire^{i\frac{2\theta+\beta}{3}}\begin{pmatrix}0&0&i\\
i&0&0\\
0&i&0\end{pmatrix}dz.$$
Finally, choose a new coordinate: $z\mapsto w = r e^{i\frac{2\theta+\beta}{3}}z$, and we obtain $$\eta=\lambda^{-1}\begin{pmatrix}0&0&i\\
i&0&0\\
0&i&0\end{pmatrix}dw.$$

 Summing up and since the gauge transformation induces  an isometry of $\C P^2$, we have
\begin{proposition}
Any vacuum can be deformed  by  gauge transformations and coordinate changes to the potential of the Clifford torus.
\end{proposition}

\section{Minimal Lagrangian immersions in $\C P^2$ with  
symmetries}
\label{Sec:symmetries}

\subsection{General background}
For all classes of surfaces, the surfaces admitting some symmetries are
of particular interest and beauty.
In this section we discuss symmetries of contractible minimal Lagrangian surfaces.

Let $M=\D$ be a contractible Riemann surface. (Note that we can exclude the discussion of $S^2$.) 
While a basic definition of a symmetry $\mathcal{R}$  for  the image  $f(\D)$ of a minimal Lagrangian surface $f: \D \rightarrow \C P^2$ may only mean 
$\mathcal{R}f(\D) = f(\D)$,
  it is very helpful, if for the given minimal Lagrangian immersion there even also exists an automorphism $\gamma$ of $\D$ satisfying
\begin{equation}
f(\gamma .z) = \mathcal{R} f(z) \hspace{2mm} \mbox{for all} \hspace{2mm} z \in \D.
\end{equation}

Note that most of our functions are functions of $z$ and $\bar{z}$. 
However, where no confusion can occur we frequently drop the argument involving $\bar{z}$.

The existence of  such an automorphism  $\gamma$ can usually be proven
 as in Theorem 2.7, \cite{DoHa;sym1},  if the induced metric is complete
(and  we assume anyway that $\D$ is simply-connected).
 
 Therefore, in this paper, for a minimal Lagrangian immersion $f: M\rightarrow \C P^2$,  a \emph{symmetry} will always be a pair $(\gamma, \mathcal{R}) \in
(\mathrm{Aut}(M), \mathrm{Iso}_0(\C P^2)),$ such that 
\begin{equation} \label{basic-symmetry}
f(\gamma\cdot z) = \mathcal{R}  f(z) \hspace{2mm} \mbox{for all} \hspace{2mm} z \in M
\end{equation}
 holds. Note that here we use the lower label $0$ to denote the connected component of the isometry group of $\C P^2$ and  that we actually have $ \mathrm{Iso}_0(\C P^2) = PSU (3)$. 
 
From here on we will always assume that $f$ is \emph{full}, i.e., that  if we have some  
 $\mathcal{R} \in   \mathrm{Iso}_0(\C P^2)$ such that
$\mathcal{R} f(p) = f(p)$ for all $p \in \D,$ then $\mathcal{R} = id$.

The following result is an easy consequence of the definitions.

\begin{lemma} \label{trafolift}
Assume $f: \D \rightarrow \C P^2$ is  a minimal Lagrangian immersion with a horizontal lift $\f$. 

\begin{enumerate}
\item If $f$ is full, then also the horizontal lift $\f$ of $f$ is full.

\item  If $(\gamma, \mathcal{R})$ is a symmetry of $f$ and $\mathcal{R} = [R]$
for some $R\in SU(3)$,  then 
$(\gamma,\delta_{\gamma,R} R)$ is a symmetry of $\f$
for some $\delta_{\gamma, R} \in S^1$ and $\delta_{\gamma,R} R$ is uniquely determined.
\end{enumerate}
\end{lemma}

\begin{proof}
\begin{enumerate}
\item  Assume that there exists some $H \in SU(3)$ satisfying $H \f(z) = \f(z)$ for all $z \in \D$. Then after projection down to $\C P^2$ we obtain $[H] f(z) = f(z)$ for all $z \in \D$. Since $f$ is full, we obtain $[H] = I$, i.e. $H = cI$ with $c^3 = 1$. But $c \f(z) = \f(z)$ for all $z \in \D$ implies $c=1$ and $H=id$, thus $\f$ is also full.

\item The equation $f(\gamma .z) = [R] f(z)$   implies that $\f(\gamma .z)$  and 
$ R \f(z)$ are horizontal lifts of $f(\gamma .z) $ and $ [R] f(z)$ respectively. Hence there exists
$\delta_{\gamma, R}$ with the properties as claimed.
\end{enumerate}
\end{proof}
Furthermore,  the following result can be shown almost verbatim as in Lemma 2.5 and Theorem 2.6 in \cite{DoHa;sym1}, respectively. 

\begin{lemma}
Let  $f: \D \rightarrow \C P^2$ be a minimal Lagrangian immersion and 
$(\gamma, \mathcal{R}) \in
\mathrm{Aut}(\D)\times  \mathrm{Iso}_0(\C P^2),$ a symmetry of $f$.
Then $\mathcal{R}$ is uniquely determined by $\gamma$.
\end{lemma}

\begin{theorem}\label{closed}
Let  $f: \D \rightarrow \C P^2$ be a minimal Lagrangian immersion and put
\begin{equation*}
\Gamma_{\D} = \left\{  \begin{array}{l}
\mbox{$\gamma \in \mathrm{Aut}(\D) 
$ such that there exists some 
$\mathcal{R} \in \mathrm{Iso}_0(\C P^2)$} \\
 \mbox{so that $(\gamma, \mathcal{R})$ is a symmetry for $f$}
 \end{array}
 \right\}.
\end{equation*}
Then $\Gamma_{\D} $ is a closed subgroup of  $ \mathrm{Aut}(\D).$
Moreover, the natural homomorphism $\Gamma_{\D}   \rightarrow  \mathrm{Iso}_0(\C P^2)$, 
$\gamma \mapsto \mathcal{R}$,  is real analytic.
\end{theorem}

The discussion of the behaviour of the horizontal lift under symmetries induced from $f$ is somewhat more complicated. First we observe that Lemma \ref{trafolift} and 
Theorem 2.6 of \cite{DoHa;sym1} yield

\begin{proposition} \label{homphi}
Let $f: \D \rightarrow \C P^2$ be a minimal Lagrangian immersion 
defined on some contractible Riemann surface
with a horizontal lift $\f$.
Then there exists a real analytic homomorphism 
$\phi : \Gamma_{\D} \rightarrow U(3)$  such that 
\begin{equation*}
\f (\gamma.z) = \phi(\gamma) \f(z)
\end{equation*}
for all $z \in \D $ and all $\gamma \in \Gamma_{\D}$.
More precisely, if $R \in SU(3)$ satisfies $[R] = \mathcal{R}$, then 
$\phi(\gamma)= c(\gamma) R$ with $c(\gamma) \in S^1.$ 
\end{proposition}

Next we consider the transformation behaviour of the frame $\mathcal{F}(\f)$ under symmetries.
We recall the definition of the frame (see \cite{MM}):
\begin{equation} 
\mathcal{F}(\f) = (  \sqrt{a}^{-1}e^{-u/2} \xi,   \sqrt{b}^{-1}e^{-u/2} \eta, \f),
\end{equation}
where $\xi=\f_z-(\f_{z}\cdot \bar{\f})\f$,  $\eta=\f_{\bar z}-(\f_{\bar z}\cdot \bar{\f})\f$,  
$a = e^{-u} \xi \cdot \bar{\xi}$, $b = e^{-u} \eta \cdot \bar{\eta}$ and  in our case $\xi = \f_z$, $\eta = \f_{\bar{z}}$, $a =b =1$,  and  $g = 2 e^{u}dzd\bar{z}$ is the induced metric.
Moreover, we assume that we have chosen the lift $\f$ such that $\det \mathcal{F}(\f) = -1 $
holds.

Let now $(\gamma,R)$ be a symmetry of $\f$.
Then a straightforward computation yields
\begin{equation} \label{transmathcal{F}}
\mathcal{F}(\f \circ \gamma)(z) = c R \mathcal{F}(\f)(z)k(\gamma, z), 
\end{equation}
where  $k(\gamma, z) \in K=U(1)$ and $c \in S^1$  as above.

More precisely we have
\begin{equation}
k(\gamma, z ,\bar z) = \di (|\gamma'| / \gamma', |\gamma'| / \bar{\gamma}',1),
\hspace{2mm} \mbox{where} \hspace{2mm}  \gamma^{\prime}=\gamma_z.
\end{equation}

In this paper the transformation behaviour of the extended frame 
$\F$ is of great importance. For this we recall that before defining $\F$ we have 
normalized $\mathcal{F}(\f)$ so as to have determinant $-1$ by multiplying by some factor in $S^1$.
Note, these factors may be different for $\mathcal{F}(\f \circ \gamma)(z)$ and
 $ \mathcal{F}(\f)(z) $  in \eqref{transmathcal{F}}.  Properly normalizing the frames we obtain
 \begin{proposition}
 Let $(\gamma,R)$ be a symmetry of $\f$. Then we can assume w.l.g.
 \begin{equation} \label{transmathcal{F2}}
\mathcal{F}(\f \circ \gamma)(z) = \hat{c} R \mathcal{F}(\f)(z)k(\gamma, z), 
\end{equation}
where $k(\gamma, z) \in K=U(1)$ as above,
but also such that $ \hat{c} \in S^1 $ satisfies $\hat{c}^3 = 1$  and we have 
$\det \mathcal{F}(\f) = -1$ and 
$\det \mathcal{F}(\f \circ \gamma) = -1.$
\end{proposition}
\begin{proof}
After multiplying the frames with the corresponding factors, equation \eqref{transmathcal{F}}
changes as stated. But by taking determinants the last claim now follows.
\end{proof}

\subsection{Symmetries of Lagrangian surfaces defined on  contractible Riemann surfaces in the loop group formalism}

Inserting the loop parameter, see subsection \ref{subsec:frames} above,  produces a family of frames, called temporarily $\mathcal{F}_{\nu},$ with initial condition 
$\mathcal{F}_{\nu} (z, \bar z) = \mathcal{F}(z, \bar z),$ and  extends  equation 
(\ref{transmathcal{F2}}) above to
\begin{equation}\label{Fgammaz}
\mathcal{F}_\nu (\gamma .z) =\widetilde{ \chi }(\nu) \mathcal{F}_\nu(z) k (\gamma, z)
\end{equation}
 with $\widetilde{\chi}(\nu) $ unitary and 
 $\widetilde{\chi}(\nu = -i) = \hat{c} R.$
 By an argument given above \eqref{F} one can even assume without loss of generality 
 $\widetilde{\chi}(\nu) \in SU(3)$, $\det \mathcal{F}_\nu (\gamma .z)  
 = -1$, $ \det \mathcal{F}_\nu (z) = -1$ and $\hat{c}^3 = 1$.

It follows easily from the definition \eqref{F} of $\F (z, \bar z, \lambda)$ that $\F$ satisfies the equation
 \begin{equation} \label{trafoF}
 \F (\gamma.z, \overline{\gamma.z}, \lambda) = \chi(\lambda) \F(z, \bar z, \lambda) k(\gamma,z, \bar z),
 \end{equation}
 with $\chi(\lambda) = \widetilde{\chi}(-i\lambda^{-3})$ and $\chi(1) = \hat{c}R$.

Actually $\F (z, \bar z,\lambda) $ induces a family
 of minimal Lagrangian immersions  $ f_{\lambda}$, by projecting the last column 
 $\f_\lambda$ of $\F(z, \bar z, \lambda) $ to $\C P^2$.
 One observes that then $\f_{\lambda}$ is a horizontal lift for  $f_{\lambda}$, and 
 $\F(z, \bar{z},\lambda)   =  \F(\f_\lambda)(z, \bar{z}, \lambda)$ holds. We point out that the matrices $\F (z, \bar z,\lambda)$ are in $\SU$.
 
With this notation  we then obtain
\begin{equation} \label{basic-lambda- symmetry}
f_{\lambda} (\gamma . z) = [\chi(\lambda)] f_\lambda(z) \hspace{2mm} \mbox{for all} \hspace{2mm} z \in \D.
\end{equation}

\section{Symmetries $(\gamma, \mathcal{R})$ with $\gamma$ or $\mathcal{R}$ of finite order}
\label{sec:finiteorder}
\subsection{Symmetries $(\gamma,R)$ of minimal Lagrangian immersions where $\gamma$ has a fixed point}

In this subsection we consider minimal Lagrangian  immersions $ f: \D \rightarrow \C P^2$
defined on a contractible Riemann surface with some symmetry
$(\gamma, \mathcal{R})$ of $f$, 
where $\gamma$ has a fixed point in $\D$. We also assume as before, that each such immersion is full.
We have

\begin{theorem}\label{thm-sym-rot}
Let  $f: \D \rightarrow \C P^2$ be a minimal  Lagrangian immersion with horizontal lift $\f$
and  $(\gamma,\mathcal{R})$ a symmetry of $f$.
Assume that $\gamma$ has a fixed point $z_0 \in \D$. 
Then there exists an extended frame $\F(\f)$, normalized to $\F(\f)(z_0,\lambda) = I$, of $f$ 
such that its Birkhoff splitting  $\F = \F_- W_+$ satisfies
\begin{equation}\label{rot-F_-}
\F_-(\gamma\cdot z, \lambda) = T \F_-(z,\lambda)T^{-1},
\end{equation}
 where $T \in K$.
Moreover, the Maurer-Cartan form $\eta$ of $\F_-$, i.e. the normalized potential  of $f$, satisfies
\begin{equation}\label{rot-eta}
\eta(\gamma\cdot z, \lambda) = T \eta(z,\lambda)T^{-1}.
\end{equation}
Conversely, if we start from some normalized potential $\eta$ satisfying
(\ref{rot-eta}) for some symmetry $(\gamma,\mathcal{R})$ with fixed point $z_0 \in \D$ of $\gamma$,
and if $\eta$ is finite at $z_0$, 
then the solution to the ODE $dC = C \eta$, $C(z_0,\lambda) = I$ satisfies (\ref{rot-F_-}). 
From this we obtain
\begin{equation}
f(\gamma\cdot z,\lambda) = T f(z,\lambda).
\end{equation}
\end{theorem}

\begin{proof}
Choose a base point $z_0$  and assume $\F (z_0,\lambda)=I$.
Evaluating \eqref{trafoF}
at $z = z_0$ we obtain 
$I  = \chi (\gamma, \lambda) k(\gamma, z_0).$
This shows
$$\chi (\gamma,\lambda) = k(\gamma, z_0)^{-1}$$
and implies that $\chi$ is independent of $z$ and
$\lambda$, denoted by $T$.
Performing  the unique  Birkhoff decomposition $\F = \F_- W_+$ with $\F_-$ of the form $\F_-(z,\lambda) = I + \mathcal{O}(\lambda^{-1})$,  we obtain  $\F_- (\gamma \cdot z,\lambda) W_+(\gamma \cdot z,\lambda) = \F(\gamma \cdot z,\lambda)  = T \F(z,\lambda) k(\gamma,z)
 = T \F_-(z,\lambda)T^{-1} \cdot T W_+(z,\lambda) k(\gamma,z)$ from which, together with 
$\F_-(z,\lambda) = I + \mathcal{O} (\lambda^{-1}),$ we infer
\begin{equation*}
\F_-(\gamma\cdot z, \lambda) = T \F_-(z,\lambda)T^{-1}, \ \ W_+(\gamma\cdot z,\lambda) = T W_+(z,\lambda) k(\gamma,z)
\end{equation*}
and the first part of the theorem follows.

To prove the converse we split $C = \F V_+$ such that the leading term $V_0$ of $V_+$ has only positive diagonal entries. Then the uniqueness  
of the Iwasawa splitting shows that $C \circ \gamma = T \F T^{-1} \cdot T V_+ T^{-1}$ implies $T \F (z,\lambda) T^{-1} = \F(\gamma \cdot z,\lambda) $ and $T V_+  (z,\lambda) T^{-1} = V_+(\gamma \cdot z,\lambda) $ and the claim follows.
\end{proof}

\begin{corollary}
Let  $f: \D \rightarrow \C P^2$ be a (full) minimal  Lagrangian immersion
and  $(\gamma,R)$ a symmetry of $f$. Assume that $\gamma$ has finite order $m$. Then $\gamma$ has a fixed point in $\D$, and the theorem applies. Moreover, if $f$ is full, then $R^m = I$ holds.
\end{corollary}

\begin{example}[Minimal Lagrangian surfaces with finite order symmetry]
Let $\gamma \cdot z=e^{\frac{2\pi i}{m}}z$ and $T=\mathrm{diag}(e^{\frac{2\pi i}{m}}, e^{-\frac{2\pi i}{m}}, 1)$.
Take the normalized potential $$\eta=\lambda^{-1}\begin{pmatrix}0&0&a(z^m)\\ z^{-3}b(z^{m})&0&0\\0&a(z^m)&0\end{pmatrix}dz,$$
 where  $a(z)$, $b(z)$ are any holomorphic functions, $b$ vanishes at $z=0$ and $m\geq 3$ is an integer.
Then $\eta$ satisfies $\gamma^* \eta=T\eta T^{-1}$. By the above theorem, $\eta$ produces a  minimal Lagrangian immersion
which possesses an $m$-fold symmetry with the fixed point $z=0$. 
This describes a new class of minimal Lagrangian immersions.
\end{example}

\begin{remark}
If one starts from some minimal Lagrangian surface with finite order symmetry, then the normalized potential has the form as above, but with meromorphic functions $a$ and $b$.
It would be interesting to know which of the potentials as above with $a$ and $b$ 
meromorphic will yield smooth immersions.
The analogous problem for CMC surfaces in $\R^3$ was solved in  \cite{DoHaMero}.
\end{remark}

Another possibility is, where $T$ has infinite order. In this case the closure of the set  $\{T^m; m \in \mathbb{Z} \}$ is a continuous group and also the closure of the set 
 $\{\gamma^m; m \in \mathbb{Z} \}$ is a continuous group. We will discuss this case in Section \ref{sec:large symmetries} below.

\subsection{Symmetries $(\gamma,\mathcal{R})$ of minimal Lagrangian immersions, where $\mathcal{R}$ has finite order, but  $\gamma$ has no fixed point}  

In the last subsection we have discussed the case, where $\gamma^m = id$ for some symmetry $(\gamma,\mathcal{R})$ of some minimal Lagrangian immersion $f$.
We have seen that in this case $\gamma$ has a fixed point, say $z_0\in \D$  and that $\mathcal{R}^m =I$ holds if $f$ is full.

In this subsection we consider the case, where $\mathcal{R}^m = I$ holds, but where $\gamma$ does not have any fixed point. In preparation for this we prove

\begin{proposition}
Let  $f: \D \rightarrow \C P^2$ be a minimal  Lagrangian immersion and put
\begin{equation}
\ker (f) = \{ \kappa \in \mathrm{Aut}(\D); f(\kappa\cdot z) = f(z) \hspace{2mm} \mbox{for all} \hspace{2mm} z\in \D\}.
\end{equation}
Then $\ker(f)$ is a discrete subgroup of $\mathrm{Aut}(\D)$ and acts freely and discontinuously 
on $\D$.
\end{proposition}
\begin{proof}
The proof can be taken almost verbatim from the proposition on p.446 of \cite{DoHa;sym1}.
\end{proof}

From this we obtain
\begin{proposition}\label{prop:4.4}
Let  $f: \D \rightarrow \C P^2$ be a minimal  Lagrangian immersion and 
$(\gamma,\mathcal{R})$ a symmetry of $f$, where $\mathcal{R}$ satisfies $\mathcal{R}^m = id$, $\mathcal{R}^{m-1} \neq id$ for some positive integer $m$. 
Then we obtain
\begin{enumerate}
\item[(a)] $\gamma^m \in \ker(f).$

\item[(b)] $f$ descends to a minimal Lagrangian immersion $\hat{f}$ from the Riemann surface
 $\hat{M} = \D / \ker(f)$ to $\C P^2$.

\item[(c)] The symmetry $(\gamma,\mathcal{R})$ of $f$ descends to a symmetry $(\hat{\gamma}, \mathcal{R})$ of $\hat{f}$ and $\hat{\gamma}$ acts on $\hat{M}$ as an automorphism of finite order $m$.
\end{enumerate}
\end{proposition}

\begin{corollary}
Under the assumptions of Proposition \ref{prop:4.4} above, if, in addition, $\gamma$ has a fixed point in $\D$, then $\gamma$ is of finite order $m$.
\end{corollary}
\begin{proof}
From part a) above we conclude $\gamma^m = id$, since $\gamma^m \in 
 \ker(f)$ has a fixed point and  $\ker(f)$ acts fixed point free.
\end{proof}

\subsubsection{\bf The case $\D = \C$ } 
In this case we can give a more precise description.
Let $\D = \C$  and  $f: \D \rightarrow \C P^2$ be a minimal  Lagrangian immersion.
Let $(\gamma,\mathcal{R})$ be a symmetry of $f$, where $\mathcal{R}$ satisfies $\mathcal{R}^m = id$, $\mathcal{R}^{m-1} \neq id$ for some positive integer $m$ and where $\gamma$ does not have a fixed point in $\C$. 

Since an automorphism of $\C$ has no fixed point if and only if it is a translation, we conclude that $\ker(f)$ is a discrete group of translations and therefore is either ${0}$, or
$\mathbb{Z} \omega$ with some nonzero $\omega \in \C$ or $\mathbb{Z} \omega_1  + \mathbb{Z} \omega_2$  with  $\omega_1$ and $\omega_2$ linearly independent over $\R$.  
Hence $\hat{M}$ of Proposition \ref{prop:4.4}is the complex plane  
or a cylinder $\mathcal{C} = \C / \mathbb{Z} \omega$ or a torus $\mathcal{T}= \C / \mathbb{Z} \omega_1  + \mathbb{Z} \omega_2$.

Moreover, $\gamma$ acts on $\hat{M}$ as an automorphism of order $m$.
Since we also have assumed that $\gamma$ does not have a fixed point in $\C$, we know 
$\gamma.z = z + \delta,$ for some nonzero $\delta \in \C$.

As a consequence, $\gamma^m . z  = z + m \delta$ is in $\ker(f)$, whence $\ker(f)$ is non-trivial and $\hat{M}$ is not simply-connected. 

Since $\gamma$ has no fixed point, only $\ker f=\Z \omega$ or $\ker f=\Z \omega_1+\Z \omega_2$ can occur. 
Thus we have

\begin{proposition}
Let $\D = \C$  and  $f: \D \rightarrow \C P^2$  a minimal  Lagrangian immersion.
Let $(\gamma,\mathcal{R})$ be a symmetry of $f$, where $\mathcal{R}$ satisfies $\mathcal{R}^m = id$,  $\mathcal{R}^{m-1} \neq id$ for some positive integer $m$ and where $\gamma$ does not have a fixed point in $\C$. 
Then $\gamma(z) = z+ \delta$ and only the following two cases occur:
\begin{enumerate}
\item[(i)] In the case when $\ker f=\Z \omega$:   $\delta=r\omega\neq 0$ with $r$ rational; 
 
 \item[(ii)] In the case when $\ker f=\Z \omega_1+\Z \omega_2$: $\delta=r_1\omega_1+r_2\omega_2\neq 0$ with $r_1, r_2$ rational.
 \end{enumerate}
\end{proposition}

The example listed in section 2.4 yields an example for this situation.

\subsubsection{The case of the (open)  upper half-plane $\D = \mathbb{H}$ }

We look for an example of a minimal Lagrangian immersion  $f: \mathbb{H} \rightarrow \C P^2$ 
which has a symmetry, $(\gamma, \mathcal{R})$, such that $\mathcal{R}^m = id$ for some positive integer $m$, but where $\gamma \in  \mathrm{Aut}(\mathbb{H}) \cong PSL(2,\R)$ does not have any fixed point in $ \mathbb{H}$.
 
In  a case as considered we know that the matrix $A\in SL(2,\R)$ representing 
the M\"{o}bius transformation $\gamma$ can be assumed to either have the double eigenvalue $1$ or two different positive eigenvalues $a_0$ and $a_0^{-1}$.

In the first case $\gamma$ acts by a translation 
parallel to the real  axis and in the second case 
by $\gamma.z = a_0^2z$ which turns into a translation parallel to the real axis 
after an application of the biholomorphic transformation $ \mathbb{H} \rightarrow \St$, $z \mapsto \ln(z) - i \frac{\pi}{2}$ with 
$\St = \R \times (-\frac{\pi}{2}, \frac{\pi}{2})$.

Note, by our choice of mapping the real axis becomes a central axis of the strip $\St$.
As a consequence, the strip $\St$ is invariant under complex conjugation.

Thus in all cases we can assume $\gamma\cdot z = z + p$, $p \in \R$, where $z$ is in $ \mathbb{H}$ or in $\St$. Transporting $f$ to the corresponding strip we obtain for all $z$:

\begin{enumerate}
\item $f(z + p) = \mathcal{R} f(z),$
\item $f(z+ mp) = \mathcal{R}^m f(z) = f(z).$
\end{enumerate}

In the case when $m = 1$, i.e. in the case where $\mathcal{R}^m =  \mathcal{R} = I$, the immersion descends to a minimal Lagrangian immersion from the cylinder
$\C / {p  \mathbb{Z}}$ to $\C P^2$. 
In all other cases one obtains a minimal Lagrangian immersion from the 
cylinder $\C / {mp  \mathbb{Z}}$ to $\C P^2$ with an additional $m$-fold symmetry.

It seems that one can determine, similar to \cite{DoKoCoarse}, what potentials one needs to choose to obtain minimal Lagrangian cylinders with an $m$-fold symmetry as described above. This may be discussed elsewhere.
Instead, we start from the strip $\St$ defined just above and consider potentials which lead to cylinders with an $m$-fold symmetry.

To this end, let $a,b, \delta: \St \rightarrow \C$ be any holomorphic functions on 
 $\St$ of period  $2\pi$ which cannot be extended beyond any point of the boundary lines of  $\St$ and are real along the real axis. 
 
 Consider the matrix differential one-form 
 $$\eta=
 \begin{pmatrix}
i\delta (z) & - \overline{ b(\bar{z})}&a( z)\\
  b(z) & - i \delta(z) & - \overline{ a(\bar{z})}\\
  - \overline{ a(\bar{z})}&a( z)& 0
  \end{pmatrix}dz.$$

 Note that $\eta$ is actually defined on all of $\St$, skew-Hermitian along the real axis and periodic of period $2 \pi$.
 Then the solution to the (matrix) ODE $dC = C \eta$, $C(z=0) = I$ 
  satisfies $C(z+ 2 \pi) = \mathcal{R} C(z)$ and is unitary along the real axis.
  
  Next we introduce the parameter $\lambda \in S^1$ as follows:
$$\eta_\lambda (z)=
 \begin{pmatrix}
 i\delta (z) & - \lambda \overline{ b(\bar{z})}& \lambda^{-1}a( z)\\
 \lambda^{-1} b(z) & - i\delta(z) & - \lambda\overline{ a(\bar{z})}\\
  - \lambda\overline{ a(\bar{z})}& \lambda^{-1}a( z)& 0
  \end{pmatrix}dz.$$
Note that $\eta_\lambda$ is contained in $\Lambda sl(3,\C)_\sigma$, is primitive 
  relative to $\sigma$ and skew-hermitian along the real axis, and  
$2\pi$-periodic.

Let $C(z, \lambda) $ denote a solution to $dC=C \eta$, $C(0,\lambda)  = I$. 
Then there exists $\chi(\lambda)\in \Lambda SU(3)_{\sigma}$ such that 
$\gamma^*C=\chi(\lambda)C$, where $\gamma$ denotes the translation by $ 2 \pi$.
Assume now that  $\chi(\lambda = 1)$ has finite order $m$, i.e. that
 $\chi(\lambda = 1)^m = I$ holds.

Following the loop group procedure (applied to the primitive $1$-form $\eta_\lambda$ above), we 
perform  the  Iwasawa decomposition $C = \F W_+$ and  obtain 
$\gamma^*\F=\chi(\lambda) \F k$ for $k\in K$.

Defining  $\f_\lambda$ as the last column of $\F_\lambda$ and $f_\lambda$ as the projection of $\f_\lambda$ to $\C P^2$ we obtain
$$\gamma^*f _\lambda =  [\chi(\lambda)] f_\lambda.$$ 
 Moreover,  $f_\lambda $ is 
a  minimal Lagrangian immersion defined on $\St.$ 
The surface $f = f_{\lambda=1}$  descends, 
because of  $ \mathcal{R}^m = [\chi(\lambda= 1)]^m= I,$  to a minimal Lagrangian immersion $$ f_{\mathcal{C}} : \St / {m 2 \pi \mathbb{Z}} \rightarrow \C P^2,$$
on which we have the symmetry $(\gamma, [\chi(\lambda = 1)] = \mathcal{R})$, and where 
$\mathcal{R}$ has finite order $m$ and $\gamma$ has no fixed point in $\St.$

\begin{remark}
In the last example a crucial step is, where we assume that  $\chi(\lambda = 1)^m = I$ holds. 
We will show in a planned continuation to this paper how one can find functions $a,b,\delta$ as above such that this additional assumption holds.
\end{remark}

\section{Minimal Lagrangian surfaces with large groups of symmetries}
\label{sec:large symmetries}

For immersed surfaces in a manifold $N$ two major types of symmetries occur: 

-- at one hand  the transformations of $N$ which leave the surface invariant, called 
 \lq\lq extrinsic symmetries\rq\rq;

-- on the other hand the group of isometries of the induced metric, called 
\lq\lq intrinsic  symmetries\rq\rq.

In this paper, the group of  extrinsic  symmetries of a minimal Lagrangian surface consists, by definition, of all
pairs $(\gamma, \mathcal{R}) \in
(\mathrm{Aut}(M), \mathrm{Iso}_0(\C P^2)),$ such that 
\begin{equation} \label{basic-symmetry}
f(\gamma\cdot z) = \mathcal{R}  f(z) \hspace{2mm} \mbox{for all} \hspace{2mm} z \in M.
\end{equation}

It thus is a closed subgroup of $(\mathrm{Aut}(M), \mathrm{Iso}_0(\C P^2))$, whence a Lie group. But also the group $\Gamma_M$ 
of first components is a closed subgroup of  
$\mathrm{Aut}(M)$, see Theorem \ref{closed} on symmetries and therefore a Lie group.

\subsection{Lie groups $\Gamma_M$ of extrinsic symmetries of dimension $ \geq 2$}
If the dimension of $\Gamma_M$ is at least two, then the surface is homogeneous.
In this case the classification is well known (see \cite{DoMaExplicit}, Section 7,  for references)

\begin{theorem}
\begin{enumerate}
\item Every minimal Lagrangian immersion $f:S^2 \rightarrow \C P^2$ is homogeneous and $f(S^2)$ is, up to isometries of $\C P^2,$ contained in $\R P^2$.

\item  If $\D$ denotes the unit disk in $\C$, then there does not exist any  homogeneous, minimal Lagrangian immersion  $f: \D \rightarrow \C P^2$ .

\item Every homogeneous minimal Lagrangian immersion $f:\C \rightarrow \C P^2$  is isometrically isomorphic with the Clifford torus.
\end{enumerate}
\end{theorem}

\subsection{Lie groups $\Gamma_M$ of  extrinsic symmetries of  dimension $1$}
We exclude from here on the homogeneous cases.

Assume now that the minimal Lagrangian immersion $f: M \rightarrow \C P^2$
 is \emph{equivariant}, i.e., it admits a  one-parameter group $(\gamma_t, R(t)) \in 
(\mathrm{Aut}(M), \mathrm{Iso}_0(\C P^2))$ of extrinsic automorphisms, 
meaning that the symmetries are induced by isometries of $\C P^2$.

Then, up to a biholomorphic change of domains and possibly a transition to the 
universal cover, one obtains exactly two types of immersions, all defined on contractible domains in $\C$.

\begin{enumerate}
\item \emph{rotationally equivariant} minimal Lagrangian immersions, i.e. those, where the one-parameter group is the full group of rotations about a point $z_0 \in \D$,

\item \emph{translationally equivariant} minimal Lagrangian immersions, i.e. those, where $\D$ can be realized as a strip and the one-parameter group as the full group of translations  (without loss of generality,  parallel to the real axis).
\end{enumerate}

Both cases have been investigated, see \cite{DoMaNewLook}, \cite{DoMaExplicit}.  We briefly summarize the main results:
First we mention \cite{DoMaExplicit}, Theorem 5.

\begin{theorem}\label{Th6.4}
Any  minimal Lagrangian immersion $f$ from $\C$ or $S^2$ into $\C P^2$ which is rotationally equivariant has a vanishing cubic Hopf differential, and therefore  is  totally  geodesic in $\C P^2$ and  its image
 is, up to isometries of $\C P^2$, contained in 
$\R P^2$.
\end{theorem}

For translationally equivariant surfaces, we obtain, see  \cite{DoMaNewLook}, Theorem 6.

\begin{theorem}
For any translationally equivariant minimal Lagrangian immersion, the extended frame $\F$ can be chosen such that
$\F (0, \lambda) = I$ and 
\begin{equation}\label{eq:equiv-F_lambda-0}
\mathbb{F}(t+z, \lambda)=\chi(t, \lambda)\mathbb{F}(z, \lambda),
\end{equation}
holds, where $\chi(t, \lambda)$ is a one-parameter group in $SU(3)$.
\end{theorem}

Using this one can apply  \cite{BuKi} and obtains, see \cite{DoMaNewLook},  Proposition 3.

\begin{theorem} \label{Th6.3}
A minimal Lagrangian surface in $\mathbb{C}P^2$ is translationally equivariant if and only it is generated by a degree one constant potential
$D(\lambda)dz$.
In this case the immersion can be defined without loss of generality on all of $\C$. 
The potential function $D(\lambda)$ can be obtained from the extended frame $\F$ satisfying \eqref{eq:equiv-F_lambda-0} and $\F(0,\lambda)=I$ by the equation
$$D(\lambda)=\F(z,\lambda)^{-1}\partial_x\F(z,\lambda)|_{z=0}.$$
\end{theorem}

The translationally equivariant minimal Lagrangian surfaces have been investigated by many authors. For a classical description see \cite{CU94}. The case of tori was investigated more generally in \cite{MM}. The description of the potential $D(\lambda) dz$ above has been given in \cite{BuKi}. A more detailed study of these surfaces in the spirit of the present paper can be found in \cite{DoMaNewLook}.

\subsection{Coarse classification of the surfaces with one-parameter groups of isometries}
In the last subsection we recalled what is known for surfaces with one-parameter groups  of extrinsic symmetries.
The question is what one can say if one only considers one-parameter groups of 
isometries.
For the case of constant mean curvature surfaces this question has been answered by \cite{Smyth} in a well-received paper. Here we prove the main result of \cite{Smyth} for our setting.

\begin{theorem}\label{Th6.5}
Let $f: M \rightarrow \C P^2$ be a minimal Lagrangian immersion in conformal coordinates
and denote by $g = 2 e^{u} dz d\bar{z}$ the induced metric.
Assume that this induced metric admits on $M$ a one-parameter group of isometries.

Then up to some biholomorphic changes of the domain and possibly a transition to the universal cover 
one can assume that either the immersion is defined on a strip  $\St$  containing the real axis, the one-parameter group can be assumed to act by translations parallel to the real axis and  $f$ is actually translationally equivariant,  or the (natural) normalized potential for $f$ satisfies 
 \begin{equation} \label{genequ1}
\eta(p_t z, q_t\lambda)=T\eta(z,\lambda)T^{-1},
\end{equation}
with $p_t z = e^{ip_0t}z, q_t \lambda = 
e^{iq_0 t}\lambda $, $p_0,q_0, t \in \R$ and  $T = \mathrm{diag}(\tau, \tau^{-1},1)$ a unitary diagonal 
matrix with last entry $1$ and $\tau = e^{it_0t}$.
\end{theorem}

\begin{proof}
In the first part of the claim we consider an immersion $f$ from a Riemann 
surface to $\C P^2$ and this immersion is conformal. Therefore, the connected component of the group of
isometries consists of holomorphic automorphisms of $M$.
We can thus apply the classification of Riemann surfaces with one-parameter groups of holomorphic automorphisms as discussed already above and obtain that there are
(up to biholomorphic equivalence and possibly transition to the  universal cover) just
two general types of one-parameter groups: groups of rotations about a point in $M$ and groups of translations in a strip (without loss of generality, parallel to the real axis).

{\bf Case 1 : The translational case:} In this case we note that the fact that the one-parameter group  consists of isometries implies that the conformal  metric factor $e^{u}$ only depends on $y$, putting $z = x + iy$.  Moreover, applying the one-parameter group of isometries to the 
(holomorphic) cubic form shows that its coefficient is constant.

Next we consider the Maurer-Cartan form $\alpha$ of our natural frame 
$\F$ which satisfies $\F(0) = I$. 
Then it is easy to verify that $\gamma_t^* \alpha = \alpha$ holds, where $\gamma_t$ denotes the one-parameter group of translations parallel to the real axis.
As a consequence, the frame $\mathcal{F}$ satisfies 
$\gamma_t^* \F = \chi_t \F$.
This shows that the immersion is translationally equivariant.

{\bf Case 2 : The rotational case:} In this case the one-parameter group of 
isometries consists of rotations (without loss of generality about the point $z=0$ which is contained in $M$). Now it is easy to verify that the conformal metric factor  only depends on $r$ and that the cubic form has as coefficient function a complex multiple of a (non-negative integer) power of $z$.

Now we consider the Maurer-Cartan form $\alpha_\lambda$ of the extended frame 
$\F(z, \bar{z}, \lambda)$. Then the $(13)$-entry of this Maurer-Cartan form picks up the factor $q_t^{-1}p_t$ by the operation of the one-parameter groups $p_t$ and $q_t$,
while the $(21)$-entry picks up the factor $q_t^{-1}p_t p_t^n,$ if the power of $z$ of the  
 $(21)$-entry is $n$. Now it is easy to show that there exists some one-parameter group 
 $\tau_t$ such that with $T(t) = \mathrm{diag}(\tau_t, \tau_t^{-1},1)$ we obtain 
 $\alpha(p_t.z, q_t.\lambda) = T(t) \alpha(z,\lambda) T(t)^{-1}$.
 
 As a consequence,  the frame $\F$ transforms like
\begin{equation}\label{transform_F}
 \F(p_t.z, q_t.\lambda) = H(t, \lambda) \F (z,\lambda) T(t)^{-1}.
 \end{equation}
 
 Note that here $H(t,\lambda) \in \Lambda SU(3)_\sigma$.
 In particular, the immersion $f$ derived from the frame $\F$ satisfies
 \begin{equation}
  f(p_t.z, q_t.\lambda) = H(t,\lambda) f(z,\lambda).
 \end{equation}

After modifying  the frame in \eqref{transform_F} by left-multiplication by a matrix independent of $z$ such that the new frame ( still denoted by the same letter) attains the value $I$ at $z =0$,  it 
 follows that $H(t, \lambda) = T(t)$ and thus is independent of $\lambda$.
 
 Now perform (at least locally near $z =0$) the (unique) Birkhoff decomposition of  
 $\F(p_t.z, q_t.\lambda) $:
 $$\F(p_t.z, q_t.\lambda) = L_-(z, \lambda) W_+(z, \bar{z}, \lambda),$$
 where $L_-(z,\lambda) = I + \mathcal{O}(\lambda)$.
 Then, since $H$ is diagonal, this implies
 $$  L_- (p_t.z, q_t.\lambda)  = T(t) L_-(z, \lambda) T(t)^{-1}.$$

Since the analogous relation holds for the Maurer-Cartan form $\eta$  of $L_-$, it follows that
the corresponding normalized potential $\eta$ satisfies equation (\ref{genequ1}).

 \end{proof}

 \begin{remark}
$(1)$ If we start from a normalized potential satisfying (\ref{genequ1}), then we will find in the next section that the corresponding surface actually does admit a one-parameter group of 
(intrinsic) isometries.

$(2)$ It is important to note that in this section we have considered rotational symmetries for which the fixed point is contained in the surface $M$ on which the isometries act.
If the fixed point is not contained in the surface $M$, then one can consider the universal cover of $M$. There the one-parameter group acts without loss of generality by translations and has a period.
\end{remark}

\section{Entire radially symmetric minimal Lagrangian immersions into $\C P^2$}
\label{Sec:entireradiallysym}

Radially symmetric  surfaces occur naturally in quantum cohomology and are discussed as immersions defined on $\C^*$ (\cite{DoGuRo, Guest, OK}), even though it would be more precise to lift the discussion to the universal cover $\C$. 

In this paper we concentrate, after a short general introduction, on surfaces defined on $\C$.

\subsection{The basic  setting and the basic formulas}
In our context we consider a minimal Lagrangian immersion 
 $f:\D_r^* \rightarrow \mathbb{C} P^2$, where $\D_r^* \equiv \{ 0 < |z| < r \} $ and write its  normalized potential in the form
 \begin{equation}\label{normeta}
\eta (z,\lambda) =  \lambda^{-1}\begin{pmatrix}
0&0&i a\\ 
i b&0&0\\
0& i a&0
\end{pmatrix}dz.
\end{equation}
  
We assume that this normalized potential satisfies the condition, encountered already above in \eqref{genequ1},
 \begin{equation}\label{genequ2}
\eta(p_t z, q_t\lambda)=T\eta(z,\lambda)T^{-1},
\end{equation}
with $p_t z = e^{ip_0t}z$, $q_t \lambda = 
e^{iq_0 t}\lambda$, $p_0$, $q_0$, $t \in \R$ and $T = \mathrm{diag}(\tau, \tau^{-1},1)$ a unitary diagonal matrix with last entry $1$ with $\tau = e^{it_0t}$.

In quantum cohomology this transformation property is called the \emph{homogeneity condition}.
We will use the same name for this condition.

After  a simple computation  we obtain

\begin{enumerate}
\item $q^{-1} p a(pz) = \tau a(z),$
\item $q^{-1}p b(pz) = \tau^{-2} b(z).$
\end{enumerate}

Since the normalized potential is (generally) meromorphic, we can consider a Laurent 
expansion of $\eta$  about $z=0$. Let $a_k$ be a non-vanishing coefficient in the Laurent expansion of $a$ and $b_n$ a non-vanishing coefficient in the Laurent expansion of $b$. 
Recall, if $a = 0$, then we do not get an immersed surface and if $b=0$, then we obtain a totally geodesic surface, i.e. an open part of $\R P^2 \subset \C P^2$.  
Therefore both these cases will be  avoided in this paper.

 Then the homogeneity condition implies
\begin{equation} \label{pqtau1}
q^{-1} p^{k+1} = \tau,
\end{equation}
and
\begin{equation}
 q^{-1}p^{n+1} = \tau^{-2}.
\end{equation}
Therefore, 
\begin{equation}
q^3 = p^{2k+n+3}.
\end{equation}
In particular,  $q_0 = \frac{1}{3}(2k + n + 3) p_0$ and 
$\tau = \exp(i{t_0 t})$ with $t_0 = \frac{1}{3}(k - n)  p_0$.
Furthermore, it is easy to verify that $a$ and $b$  can contain at most one pole:

Since we assume that the normalized potential is meromorphic, the condition above implies that the set of singularities  of $\eta$  in $\D_r^*$ can only consist of the origin.
Therefore,  altogether, the normalized potential $\eta$ is actually defined on $\mathbb{C}^*$ and has the form
\begin{equation} \label{etarad}
\eta (z,\lambda) =  \lambda^{-1}\begin{pmatrix}
0&0&i a_k z^k\\ 
i b_n z^n&0&0\\
0&i a_k z^k&0
\end{pmatrix}dz.
\end{equation}

After conjugation by a diagonal matrix with entries in $S^1$  and a scaling of the coordinate system
if necessary, we obtain

\begin{lemma}\label{normalized apsi}
In (\ref{etarad}) one can assume without loss of generality,  $a_k > 0$ and that $b_n$ is of the form
\begin{equation} \label{specific}
b_n = - {a_k}^{-2} \psi_0
  \hspace{2mm} \mbox{with} \hspace{2mm} \psi_0 < 0.
\end{equation}
\end{lemma}

With this notation  the cubic form is given by 
\begin{equation} \label{cubicrad}
\psi (z) = \psi_0 z^{2k + n}.
\end{equation}

Hence we obtain
\begin{theorem}
Let $\eta$ be a normalized potential for a minimal Lagrangian surface defined on 
$\D_r^*$ which satisfies the homogeneity condition (\ref{genequ2}) with 
$p_t z = e^{ip_0t}z$, $q_t \lambda = 
e^{iq_0 t}\lambda$, $p_0$, $q_0$, $t \in \R$ and $T$ a unitary diagonal matrix 
 $T = \mathrm{diag}(\tau, \tau^{-1},1)$, with $\tau = \exp(i{t_0 t})$.
 Then 
 \begin{itemize}
\item[a)] $q_0 = \frac{1}{3}(2k + n + 3) p_0$ and  $t_0 = \frac{1}{3}(k - n)  p_0$.

\item[b)] All coefficients of $\eta$ are complex multiples of some powers of $z$, this means that $\eta$ can be defined on  $\C^{*}$ and  $\eta$ has the form \eqref{etarad}, 
where one can also assume the properties are specified in \eqref{specific}.

\item[c)] The cubic form is a complex multiple of $z$, more precisely we have \eqref{cubicrad}.
\end{itemize}

Conversely, given a normalized potential of the form \eqref{etarad} and satisfying \eqref{specific},
define $p_t$, $q_t$ and  $T(t)$
with coefficients as in \rm{a)} above, then $\eta$ satisfies the homogeneity condition with these functions.
\end{theorem}

\begin{remark}
Obviously, an interesting special case is the case, where $T(t) \equiv I$ for all $t$.
The formula above for $t_0$ shows that this is equivalent to $k = n$. But in this case 
the substitution $dw = z^n dz$  reduces this special case to the one, where the normalized 
potential is constant.  
We will discuss this case in more detail at the end of this paper.
\end{remark}

An inspection of the homogeneity condition shows that this transformation describes a behaviour which is very different from the behaviour under an \emph{extrinsic symmetry}, i.e. one induced by some isometry of $\C P^2$. The transformation behaviour stated by the homogeneity condition is therefore also called \emph{intrinsic symmetry}.

However,  if we can find a $\hat{t}$, for which we obtain $q(\hat{t}) = 1$, then the homogeneity condition describes an extrinsic symmetry for the fixed value $\hat{t}$ of the variable $t$.

Frequently, radially symmetric surfaces admit a finite set of extrinsic symmetries.

\begin{corollary} \label{outersymmetry}
With the notation of the theorem above we put $\hat{t}  = 2 \pi / {q_0}$, so that $q(\hat{t}) = 1$.
Then  

a) $\hat{t}   = \frac{6 \pi}{(2 k + n + 3) p_0}.$ 

b) The exponent of $p(\hat{t})$ is $\frac{6 \pi i}{2k + n + 3}$.  

c)  The exponent of $\tau(\hat{t})$ is $\frac{6 \pi (k-n) i}{2 k + n + 3}$. 

In particular, the transformations $p(\hat{t}) $ and $T(\hat{t})$ have finite order such that
\begin{equation}
\eta(p(\hat{t})z, \lambda) = T(\hat{t}) \eta(z,\lambda) T(\hat{t})^{-1}
\end{equation}
for all $z \in \C^* $  and $\lambda \in S^1$.
\end{corollary}


\subsection{ The transformation formulas for entire radially symmetric minimal Lagrangian surfaces}

The case, where there is  a pole at $z=0$, has been treated in some cases (different from minimal Lagrangian surfaces), \cite{DoGuRo}, \cite{Guest-Lin-Its}.
It is beyond this paper to generalize the work mentioned just above.
Therefore, for the rest of this paper we will restrict to holomorphic normalized potentials defined on  $\C$, in particular defined at $z=0$.

Using that the $(13)$-entry of the normalized potential of a minimal Lagrangian immersion never vanishes (as pointed out above again), we infer

\begin{lemma}\label{lemma7.4}
If $\eta$ is the normalized potential of a  minimal Lagrangian surface defined on all of $\C$ satisfying the homogeneoeity condition \eqref{genequ2}, then all 
entries of $\eta$ are  complex multiples of a power of $z$ and  the surface 
constructed from $\eta$ is an immersion on $\C^*$. 
Moreover,  writing the $(13)$-entry of $\eta$ in the form
$i a_kz^k$, it follows that the surface constructed from $\eta$ is also an immersion at $z = 0$ if and only if $k=0$ and $a_{k} \neq 0$.
\end{lemma}

\begin{definition}
In what follows we will call surfaces defined on all of $\C$ \emph{entire}. The surfaces considered in Lemma   \ref{lemma7.4} thus will be called \emph{entire radially symmetric surfaces} and, in case they are globally immersions, \emph{entire radially symmetric immersions.}
\end{definition}

\begin{remark}
Applying Corollary \ref{outersymmetry} to entire radially symmetric 
immersions for which the cubic form is constant, we observe that then $k = n = 0$ and $p(\hat{t}) = 1$ and  $T(\hat{t}) = I$.
\end{remark}

For an entire radially symmetric surface  we  choose $z=0$ as a base point and solve the ODE $dC = C \eta$ with initial condition $C(0,\lambda) = I$ on  $\C$. 
The (unique) Iwasawa decomposition $C = \hat{\mathbb{F}} V_+$ then yields a frame which is globally defined on $\C$. 

Moreover, induced from the homogeneity condition of the normalized potential,
 $C$ satisfies the transformation formula
\begin{equation} \label{homogC}
 C(pz,q\lambda) = T C(z,\lambda) T^{-1},
 \end{equation}
and $\hat{\mathbb{F}}$, obtained from $C$ via Iwasawa decomposition  satisfies
\begin{equation} \label{homogF}
 \hat{\mathbb{F}}(pz,q\lambda) = T \hat{\mathbb{F}} (z,\lambda) T^{-1}. 
 \end{equation}
Since we have chosen the unique Iwasawa splitting, we have 
 $\hat{\mathbb{F}} (0,\lambda) =I$ and  $V_+ (0,\lambda) =I$ and also
 \begin{equation} \label{homogVplus}
 V_+(pz,q\lambda) = T V_+ (z,\lambda) T^{-1}.
 \end{equation}

Moreover, from (\ref{homogF}) we obtain for the Maurer-Cartan form $\hat{\alpha}$ 
of $\hat{\mathbb{F}}$:
\begin{equation} \label{homogalpha}
\hat{\alpha}(p_t z, q_t\lambda)=T\hat{\alpha}(z,\lambda)T^{-1},
\end{equation}
where $p,q$ and $T$ are as for $\eta$.

Note  that any two frames are gauge equivalent by some element in $K$.  
Hence, at all immersion points of $f$ we have  
$\hat{\mathbb{F}} (z,\lambda)  = \mathbb{F} (z,\lambda) k(z)$, where $\hat{\mathbb{F}}$ is as above and $\mathbb{F}$ denotes our standard normalized frame of the associated minimal Lagrangian surface.
From the Maurer-Cartan form $\alpha$ of $\mathbb{F}(z, \lambda) $ and also from
the one of 
$ \hat{\mathbb{F}} (z,\lambda)$ one can read off the metric by taking absolute values
 of the $(13)$-entry.  The matrix $k$ then has as first diagonal entry the inverse of the phase factor of the  $(13)$-entry.  From this  the Hopf differential can be computed.

\begin{theorem}
If  the normalized potential for the entire minimal Lagrangian surface
$f:\C \rightarrow \mathbb{C} P^2$ satisfies the homogeneity condition, then the metric only depends on the radius and the Hopf differential $\psi$ is of the form $\psi(z) = \psi_0 z^l$ 
with $ l = 2k + n$  and $\psi_0$ a complex number.
\end{theorem}

Since for the surface we have $f =  [\f] ,$ where $\f$ is the last column of
 $\F(z, \lambda)$, we obtain for the associated family of the surface constructed from $\eta$ the transformation formula:
 \begin{equation} \label{radialfpos}
  f(pz,q\lambda) = [T] f(z,\lambda).
  \end{equation}
 Note that this transformation formula is different from the one occurring for extrinsic symmetries.
 
More generally, one can consider minimal Lagrangian surfaces, 
 say defined on $\C$,  for which the associated family satisfies the transformation formula
 \begin{equation} \label{radialfgen}
 f(pz,q\lambda) = \mathcal{G}(t,\lambda) f(z,\lambda)
 \end{equation}
for some isometry $ \mathcal{G}(t,\lambda) \in PSU_3$.
In this paper we will only consider the condition \eqref{radialfpos} and call the surface \emph{radially symmetric}.

\subsection{Entire surfaces satisfying 
$ f(p_t z,q_t \lambda) = [T_t] f(z,\lambda)$ are entire radially symmetric}
Entire radially symmetric surface means satisfying  \eqref{radialfpos} and defined on $\C$.
Converse to the above discussion and in particular to Lemma \ref{lemma7.4} we prove
\begin{theorem}\label{Th7.6}
Let  $f:\C \rightarrow \mathbb{C} P^2$ be a  minimal Lagrangian immersion
satisfying the transformation formula \eqref{radialfpos}. So $f$ is radially symmetric by the above definition. Furthermore, the frame $\mathbb{F}$ is  globally smooth 
and such that \eqref{homogF} holds and its normalized potential satisfies the homogeneity condition.
Moreover, the normalized potential is defined on all of $\C$ and has the form 
\eqref{etarad} with $k=0$ and $a_{0}\neq 0$. 
\end{theorem}

\begin{proof}
We will carry out the proof in several steps: 

{\bf Step 1:} We start from some minimal Lagrangian surface defined on  $\C$ satisfying 
 $$f(pz,q\lambda) = [T] f(z,\lambda)$$ 
 with $p$, $q$, $T$ as above.
 Then we consider a horizontal lift $\f$ of $f$ and form the frame $\mathbb{F}$ and can assume that it attains the value $I$ at $z=0$. 
It is straightforward to verify now that the metric induced by $f$ only depends on $r^2$ and it is the restriction of a never vanishing real analytic map defined on $\R^2$.
 Let $\alpha$ denote the Maurer-Cartan form \eqref{eq:mathbbF}
  of $\mathbb{F}$ and observe that we know its entries.
Thus the frame is smooth.
 
 {\bf Step 2:}  We need to determine the cubic form in more detail.
 For this we observe that the Tzitzeica equation, written relative to polar coordinates, 
 only contains the metric, respectively its exponent actually, $u(r^2)$, and ${|\psi(z)|}^2$.
 As a consequence, $|\psi(z)|^2$ only depends on $r^2$. Now holomorphicity of $\psi$
 shows that therefore $\psi$ is of the form $\psi(z) = \psi_0 z^m$ for some non-negative integer $m$.
 
{\bf Step 3:} Next we choose one-parameter groups $p_t,q_t,T_t$ of the form as before, but with still unknown parameters $p_0$, $q_0$ and $t_0$.
 Then we write out $\alpha(pz, q\lambda)$ and $T \alpha(z,\lambda)T^{-1}$. By inspection of these two expressions it is easy to verify that given $p_0$ one can choose $q_0$ and $t_0$
 such that (also using the normalization $I$ at the base point $z =0$)
  $$\mathbb{F}(\f)(pz,q\lambda) = T \mathbb{F}(\f)(z,\lambda) T^{-1}$$
 holds.

 {\bf Step 4:} Now we perform the (unique, where possible,) Birkhoff decomposition 
 $ \mathbb{F} = \mathbb{F}_- V_+$ with $\mathbb{F}_-  = I + \mathcal{O}(\lambda^{-1})$ and conclude (since $T$ is unitary and also in $\Lambda^+ SL(3,\C)$)
 \begin{equation}\label{transformF-}
 \mathbb{F}_- (pz,q\lambda) = T\mathbb{F}_- (z, \lambda) T^{-1}.
 \end{equation}
Since we have chosen $z=0$ as a base point, the frame $\F$ is smooth around $z=0$ and attains the value $I$ there. Therefore the Birkhoff decomposition is analytic in a
 neighbourhood of $z=0$ and, as a consequence, the normalized potential is smooth at $z=0$.
Equation \eqref{transformF-} implies that the normalized potential, i.e. the Maurer-Cartan form $\eta $ of  
$ \mathbb{F}_- $   satisfies the homogeneity condition
\begin{equation*}
\eta(p_t z, q_t \lambda)=T\eta(z,\lambda)T^{-1}.
\end{equation*}
This finishes the proof of 
the theorem. 
\end{proof}

\subsection{Properties of metric and cubic form of an entire radially symmetric minimal Lagrangian surface}

In the last subsection we have seen,  how one can go from entire radially symmetric surfaces to
 normalized potentials satisfying a homogeneity condition and that  the converse procedure also holds.
 Moreover,  we have stated the transformation behaviour of frames and potentials in each step.

In this subsection we will  discuss how the metric and the cubic form 
characterize entire radially symmetric surfaces.

Theorem \ref{Th7.6} shows

\begin{corollary}
If $f: \C \rightarrow \C P^2$ is an entire radially symmetric immersion, i.e., $f$ satisfies \eqref{radialfpos}) then its cubic form 
$\Psi = \psi dz^3$ is of the form $\psi = \psi_0 z^{m}$, where we can assume 
without loss of generality $\psi_0 < 0$ and that $m$ is a non-negative integer. Moreover,  its metric only depends on $r^2$.

\end{corollary}

\begin{proposition}
Furthermore,  any minimal Lagrangian surface defined on $\C$ with cubic form 
$\Psi =  \psi_0 z^{m} dz^3,$  $\psi_0 < 0,$ and $m$ a non-negative integer, 
and which has a metric only depending on $r^2$  is entire radially symmetric.
\end{proposition}

\begin{proof} 
To prove the converse, we consider the Maurer-Cartan form of the given surface and 
one-parameter groups $p_t, q_t$ and $T(t)$ satisfying
such that the Maurer-Cartan form $\alpha$, gauged by some  
$\mathrm{diag} ( \frac{z^v}{|z|^v}, \frac{|z|^v}{z^v}, 1)$,    has the transformation behaviour
 $\alpha (p_t.z, q_t.\lambda) = T(t) \alpha(z,\lambda) T(t)^{-1}$. From this one derives the transformation behaviour for the normalized potential $\eta$ of the given immersion and observes that $\eta$ satisfies the homogeneity condition.
 \end{proof}

\begin{corollary}
Let $\Psi =  \psi_0 z^m dz^3$, $m \geq 0$, $m \in \mathbb{Z}$ be a cubic form defined on $\C$
and let $e^{u}$ be a non-negative  function on $\C$, positive on $\C^*$ and only depending on $r$. 

Assume, moreover, that these functions satisfy the Tzitzeica equation.
Then the standard Maurer-Cartan form built with $\psi = \psi_0 z^m$ and $e^{u(r)}$ produces an entire radially symmetric minimal Lagrangian surface.  
\end{corollary}

\begin{proof}
The given information yields a minimal Lagrangian surface defined on $\C$, possibly with a singularity at $z=0$. Now the theorem above applies. 
\end{proof}

\subsection{The Painlev\'{e} equation for the metric of entire radially symmetric minimal Lagrangian surfaces}

In this subsection we write out, how the Tzizeica equation for general minimal Lagrangian surfaces specializes to the case of entire radially symmetric minimal Lagrangian surfaces.
First we use the fact that the metric only depends on $r$. Thus we obtain in polar coordinates
\begin{equation} \label{Tzitzeicapolar}
u'' + \frac{1}{r} u' + 4 e^{u} - 4 |\psi|^2 e^{-2u} = 0.
\end{equation}

\begin{remark} For a normalized potential as in (\ref{etarad}) the corresponding radially symmetric surface has as cubic form a  complex multiple of $z^{2k + n}.$ The equation just above thus is the classical Tzizeica equation (in polar coordinates) if and only if $\eta$ is a constant matrix, i.e. $k = n =0$. This was missed to observe in \cite{OK}. 
\end{remark}

Using that the coefficient of the cubic form is a multiple of a power of $z$ we will show that the
equation above can be rewritten (after some substitutions) as a Painlev\'{e} equation:

Consider the function $h(s) = e^{u(r(s))} s^j$ with $s = r^l $. 
Then we also have $h(s(r)) = e^{u(r)} r^{jl}$. Hence 
$$u(r) = \log h(s(r)) - jl \log r.$$

Substituting this into the Tzitzeica equation, written  in polar coordinates,  a straightforward computation yields (with $ h = h(s(r))$ and denoting the derivative for $s$ by  a dot):
$$ \ddot h = \frac{(\dot h)^2}{h} - \frac{1}{r^l} \dot h - \frac{4}{l^2} \frac{h^2}{ r^{2l + jl - 2}}
+  \frac{4 |\psi_0|^2}{l^2} 
\frac{1}{r^{2l -2jl -2 - 4k - 2n}}  
\frac{1}{h},$$
with $\psi=\psi_0 z^{2k+n}$ and $\psi_0\in \mathbb{C}$.

We want to compare this to  the celebrated Painlev\'{e} III equation 
\begin{equation} \label{Painleve}
y^{\prime \prime}(s) = \frac{ y^\prime (s)^2}{y(s)} - \frac{y^\prime (s)}{s} +
\frac{ \alpha y(s)^2 + \beta }{s} + \gamma y(s)^3 +  \frac{\delta }{y(s)}.
\end{equation}

It is easy to verify that in the classification of Okamoto \cite{OKSO} our equation  is of  type $D_7$,
if we  choose the parameters $j$ and $l$ such that
$$ l + jl = 2 \hspace{2mm} \mbox{and} \hspace{2mm} 2l -2jl -2 -4k - 2 n = 0$$
hold, equivalently 
$$ l = \frac{1}{2} (2k + n + 3)  \hspace{2mm} \mbox{and} \hspace{2mm} jl = \frac{1}{2} (1 - 2k - n).$$
Therefore, altogether we obtain:
\begin{theorem} 
The metric of an entire radially symmetric minimal Lagrangian surface 
with metric $e^{u(r)}$ and cubic form $\psi(z) dz^3 = \psi_0 z^{2k + n} dz^3$
only depends on $r$ and satisfies the Painlev\'{e} equation PIII of type $D_7$, with $h = h(s)$ satisfying
\begin{equation} \label{PIII}
 \ddot h = \frac{(\dot h)^2}{h} - \frac{\dot{h}}{s} - \frac{16}{(2k + n +3)^2} \frac{h^2}{ s}
+  \frac{16 |\psi_0|^2}{(2k +n +3)^2} 
\frac{1}{h}.
\end{equation}
\end{theorem}

In the discussion of Painlev\'{e} equations the asymptotic behaviour of solutions is of great importance (see, e.g. \cite{FIKN}).
In our setting it is easy to determine the asymptotic behaviour of the  solutions discussed in this section at $z =0$.

\begin{theorem}
Let $f:\C \rightarrow \C P^2$ be an entire radially symmetric minimal Lagrangian surface with metric $e^{u(r)}$ and cubic form $\psi(z) dz^3 = \psi_0 z^{2k + n} dz^3$ and let 
 \begin{equation}\label{substitution}
 h(s) = e^{u(r(s))} s^j,
 \end{equation}
  with $s = r^l $, where 
 $ l = \frac{1}{2} (2k + n + 3)  \mbox{and} \hspace{2mm} jl = \frac{1}{2} (1 - 2k - n)$.
 Then for 
 $s \rightarrow 0$ the function $h(s) $ has the asymptotic behaviour 
\begin{equation} \label{asymph}
\log (h(s) )  \approx    \frac{2k -n + 1}{2k + n + 3} \log s + 2\log |a_k|+o(s). 
\end{equation}
\end{theorem}

\begin{proof}

To understand the behaviour of $h$ near $s=0$ we use one more relation, namely $C = \mathbb{F} V_+.$ Here we can assume without loss of generality that the Maurer-Cartan form of $\mathbb{F}$
has the form \eqref{eq:mathbbF}.
This leads  to the equation 
\begin{equation}\label{eq:7.21}
 a_k z^k = i e^\frac{u}{2} v_0^{-1},
\end{equation}
where $v_0$ is the $\lambda^0-$term of the $11$-entry of $V_{+}$.
Altogether we obtain 
$$h(s) = e^{u(s)} s^j = |v_0 (r(s))|^2 r(s)^{2k} s^j$$
and, since $v_0 (0) = 1$  we obtain (\ref{asymph}).
\end{proof}

The following result clarifies the relation between entire radially symmetric minimal Lagrangian surfaces and solutions to the Painlev\'{e} equation PIII with a certain asymptotics at $z=0$.
In fact, $h$ is the unique solution with the asymptotic behavior  \eqref{asymph}. 

\begin{theorem}
Each entire radially symmetric minimal Lagrangian surface with normalized potential (\ref{etarad}) yields a solution to the third Painlev\'{e} equation (\ref{PIII}) which has the asymptotics (\ref{asymph}) at $z = 0$.

Conversely, assume we have two solutions to  the third Painlev\'{e} equation 
(\ref{PIII}) which has the asymptotics (\ref{asymph}) at $z = 0$. Then these two solutions are equal. 
\end{theorem}
\begin{proof}

We only argue for the second part of the statement. Assume $u$ and $\hat{u}$ are two solutions to the third Painlev\'{e} equation (\ref{PIII}) which has the asymptotics (\ref{asymph}) at $z = 0$. Then 
from the Painlev\'{e} equation we read off $2k + n = 2 \hat{k} + \hat{n}$ and $|\psi_0|=|\hat{\psi}_0|$.
Equation \eqref{asymph} permits now to conclude $2k - n = 2 \hat{k} - \hat{n}$ and we can read off $|a_k|=|\hat{a}_k|$.
The two integer equations imply $k = \hat{k}$ and $ n = \hat{n}$. Hence we know that the normalized potentials of the two entire radially symmetric minimal surfaces associated with the two solutions started from have equal powers of $z$ at corresponding entries. But it follows from Proposition \ref{normalized apsi} that after conjugation by a diagonal matrix with entries in $S^1$  and a scaling of the coordinate system
(if necessary) permit to assume without loss of generality that the coefficients in the normalized potential satisfy $a_0>0$ and $\psi_{0}<0$ in (\ref{etarad}). Therefore the two solutions come from the same normalized potential and thus the surfaces are equal. Then also the solutions of the third Painlev\'{e} equation PIII  coincide.
\end{proof}

\begin{remark}
\begin{enumerate}
\item  In general, solutions to Painlev\'{e} equations have many singularities along the positive real axis. It is a rare, but important, case, when one finds a solution without singularities.
Since our potential is holomorphic on $\C$, and since the Iwasawa decomposition is, in our case, global, the surface is defined on all of $\C$, with a branch point at $z =0$ if $k>0$.
 In view of our substitutions \eqref{substitution} it follows, that the  solution to our Painlev\'{e} equation is smooth for $s >0$ and approaches $0$ as $s$ tends to $0$ if $2k-n+1$ is positive.

\item In the proof above we have observed that the  $(11)$-entry of the leading term $V_+$ in the Iwasawa decomposition $C = F V_+$, called $e^b$ in \cite{OK},  is a complex multiple of $z^{-k} e^{u/2}$.
\end{enumerate}
\end{remark}

\section{Examples of  entire  radially symmetric minimal Lagrangian immersions into $\C P^2$}
\label{Sec:radially symmetric}

In the last section we have explained how one can construct all 
radially symmetric minimal Lagrangian surfaces which have a normalized potential which is holomorphic on an open disk $\D_r$ in $\C$ with center $z =0$.
In this section we will discuss some very special cases following the method discussed in this paper.
For more details and explanations respectively see subsection \ref{subsection:loop method}.

In both cases we start from some normalized potential $\eta$, see \eqref{normeta}.
Then we solve the ODE
$$dC(z,\lambda)  = C (z,\lambda) \eta (z,\lambda), \quad C(0,\lambda) = I.$$ 

In the next step we perform an Iwasawa decomposition
$$ \F(z,\bar z, \lambda)=C(z, \lambda)V_{+}(z,\bar{z},\lambda).$$

One frequently mentions here the \lq\lq unique Iwasawa decomposition\rq\rq, obtained by requiring that the diagonal terms of the leading coefficient of $V_+$ are positive.

For concrete computations it is, however, sometimes useful to choose another Iwasawa decomposition, namely the one, where the Maurer-Cartan form of $ \F$ has the form
(\ref{eq:mathbbF}). This one can do without loss of generality and in this subsection we will assume this choice.

The last step in the construction procedure is to choose the last column $\f$ of $\F$ and to project it to $\C P^2$. Then the resulting map $f$ is a minimal Lagrangian surface in $\C P^2$ and $\f$ is a horizontal lift of $f$  assuming that the $(13)$-entry of 
$\eta $ never vanishes. (If this entry vanishes at some point, then one can sometimes  change this by a singular gauge so that one actually even then will obtain a minimal immersion. But in this paper we have not considered singular gauges.) 


For our example the following result is of crucial importance.
\begin{theorem}\label{Thm8.1}
Let $f: \C \rightarrow \C P^2$ be a full minimal Lagrangian immersion. The normalized potential $\eta$ of  $f$ is constant  if and only if there exists 
a one-parameter subgroup $p_t \in S^1$ such that $\eta(p_t z,p_t \lambda)=\eta(z, \lambda)$.
\end{theorem}

\begin{proof}
For a constant normalized potential $\eta=\lambda^{-1}A dz$ with constant $A$, there exists one-parameter group $p_t=\exp(itp_0)$ and $q_t=\exp(itq_0)$ such $\eta(p_t z, q_t\lambda)=\eta(z,\lambda)$ with $k=n$ (and $T=I$). It follows that $f$ is an immersion that $k=0$. Hence
we have 
$$\eta(p_t z, p_t\lambda)=\eta(z,\lambda)$$
for any one-parameter group $p_t=\exp(itc)$ with an arbitrary real number $c$.

Conversely, suppose that the normalized potential of a minimal Lagrangian immersion $f:\C \rightarrow \C P^2$ satisfies 
 $\eta(p_t z, p_t \lambda)=\eta(z,\lambda)$
 for $p_t=\exp(itc)$ with a real number $c$.

This implies that $\eta$ satisfies the homogeneity condition with $T=I$, hence $\eta$ has of form \eqref{etarad} with $k=n$. 
Since $f$ is an immersion on all of $\C$, it follows from Lemma \ref{lemma7.4} that 
$k=0$.  
Thus the normalized potential $\eta=\lambda^{-1}Adz$ is  constant.
\end{proof}

\begin{remark}
In \cite{Bobenko-Its} an initial value problem for $dC = C \eta$  was considered, where not only $I$ as initial condition is used, but any positive constant diagonal matrix. In our case,
we can always chose without loss of generality the matrix $I$ as initial condition.
In fact, for a given \emph{positive diagonal matrix} $A$, $C(z,\lambda)$ is a solution to $dC = C \eta$, $C(0,\lambda) = A$ if and only if $C_A=ACA^{-1}$ is a solution
to $dC_A = C_A \eta_A$,   $C_A(0,\lambda) = I$
for the potential $\eta_A = A \eta A^{-1}$.
\end{remark}

\subsection{Implications for metric and cubic form}

\begin{lemma}\label{lemma8.2}
If the normalized potential of a full, minimal Lagrangian immersion $f: \C \rightarrow \C P^2$ satisfies 
 $\eta (p_t z, p_t \lambda)=\eta (z, \lambda)$ for a one-parameter group $p_t\in S^1$, then 
$\psi$ is constant and the metric only depends on $r$. 
\end{lemma}
\begin{proof}
It follows from Theorem \ref{Thm8.1} that the normalized potential of $f$ is constant. 
This implies
\begin{equation} \label{genequiframe}
\F(p_t z, p_t \lambda) = \F(z,\lambda) 
\end{equation}
for the frame of the minimal Lagrangian immersion associated to $\eta$.

The invariance  (\ref{genequiframe})  of $\mathbb{F}$ implies immediately that the Maurer-Cartan form $\alpha=\mathbb{F}^{-1}d\mathbb{F}$ satisfies
\begin{equation}
\alpha(p_t z, p_t \lambda) = \alpha(z,\lambda).
\end{equation}
From this one reads off (again) that $\psi$ is constant, but also that 
the metric factor $2e^{u}$ of $f$ only depends 
on $r^2 = z \bar{z}$.
Explicitly, it follows from 
$\lambda^{-3}\psi(z) dz^3=(p_t \lambda)^{-3} \psi(p_t z) d(p_t z)^3,$
that $\psi(p_t z)=\psi(z)$
holds.
Since $\psi(z)$ is holomorphic in $z$, $\psi(z)$ is constant.
Similarly, 
$e^{u(p_t z)}=\F_{z}(p_t z) \overline{\F_z(p_t z)}=e^{u(z)}$
implies $u(p_t z)=u(z)$. Hence $u$ only depends on the radius $r^2= z \bar{z}$.
\end{proof}

\begin{remark} The conditions  $\eta (p_t z, p_t \lambda)=\eta (z, \lambda)$,  $C (p_t z, p_t \lambda)=C (z, \lambda)$,  $\F (p_t z, p_t \lambda)=\F (z, \lambda)$ and  $f (p_t z, p_t \lambda)=f (z, \lambda)$  are equivalent to each other.
Moreover, these conditions imply $f(\lambda^{-1}z,1) = f(z, \lambda)$. Thus the surface 
has the same image for all $\lambda$, but a different parametrization.
\end{remark}

\begin{remark}
We will state explicitly the Painleve III equation and its asymptotics at $0$  in equations
\eqref{PIII8} and \eqref{PIIIasymp8} below.
\end{remark}

We have started with a radially symmetric minimal Lagrangian immersion satisfying the 
special condition
\begin{equation} \label{spacialmLi}
f(e^{ip_0t} z,e^{ip_ot} \lambda)=f(z, \lambda). 
\end{equation}

In this case we also know
\begin{equation} \label{specialmLi8}
u(z, \bar{z}) = u(r).  
\end{equation}

In \cite{FrWe} the authors have considered surfaces in $\R^3$ which are
intrinsically surfaces of revolution. By definition this means that the induced metric is conformal with metric factor only depending on $r$.
Under the additional assumption that 
\begin{itemize}
\item[a)] the principal curvatures only depend on the radius,

\item[b)] the principal curvature directions only depend on the angle of rotation (but  the radius),
\end{itemize}
they proved that these surfaces have constant mean curvature and are Smyth surfaces
(for a precise formulation see loc.cit Theorem 1.3).
In the minimal surface case these surfaces are Enneper type surfaces (see loc.cit Theorem 1.5).
Smyth had assumed radially symmetric and constant mean curvature. In this regard the paper \cite{FrWe} makes an additional statement about the minimal case.

In our case one can consider the special additional assumptions:

i) $\psi = 0$,

ii)  $\psi  \neq 0$ but $a =  b = 1$ and $\psi_0 = 1$,  

or,

iii) $\psi \neq 0$ and $|a|\neq |b|$.

We will show below how we can describe these surfaces in our formalism.

\subsection{Case $\psi\equiv 0$} 

In this case the Maurer-Cartan form of the extend frame $\F(z,\bar{z},\lambda)$ is 
$\alpha=\alpha^{\prime}dz+\alpha^{\prime\prime}d\bar{z}$ given by \eqref{eq:mathbbF} with $\psi\equiv 0$.
Assume $\F(z,\bar{z},\lambda)=C(z,\lambda)V_{+}(z,\bar{z},\lambda)$ is the Iwasawa decomposition.
Then $\alpha^{\prime\prime}=\F^{-1}\partial_{\bar z}\F=V_{+}^{-1}\partial_{\bar z}V_{+}$.
Considering $z$ as a parameter, we need to solve the following $\bar\partial$-problem:
\begin{equation*}
V_{+}^{-1}\partial_{\bar z}V_{+}=\alpha^{\prime\prime},\quad
V_{+}(z,0,\lambda)=I.
\end{equation*}
Performing the following gauge transformation
$$\hat{V}_{+}=V_{+}{\mathrm{diag}}(e^{\frac{u(z,\bar{z})}{2}-\frac{u(z,0)}{2}}, e^{\frac{-u(z,\bar{z})}{2}-\frac{u(z,0)}{2}}, 1)=:V_{+}D_1,$$
we obtain
\begin{equation*}
\hat{V}_+^{-1}\partial_{\bar z}\hat{V}_{+}=i\lambda e^{u(z,\bar{z})-\frac{u(z,0)}{2}} N_{+},
\end{equation*}
where $N_{+}=\begin{pmatrix}0&0&0\\0&0&1\\1&0&0\end{pmatrix}$ 
is nilpotent.
Thus 
$$\hat{V}_+(z,\bar{z},\lambda)=\exp(\lambda h_+(z,\bar{z})N_+)$$
with $h_+(z,\bar{z})=i\int_0^{\bar z} e^{u(\xi,\bar{\xi})-\frac{u(\xi,0)}{2}}d\bar{\xi}$
and $V_+=\hat{V}_+ D_1^{-1}$.

Noticing that $V_{+}(z,0,\lambda)=I$, we conclude $\F(z,0,\lambda)=C(z,\lambda)$, thus 
\begin{equation}\label{eq:Iwasawa}
\F (z,\bar{z},\lambda)=\F(z,0,\lambda)V_+(z,\bar{z},\lambda).
\end{equation}

Similarly, consider 
$$\F_{-}(z,\lambda)=\F(z,0,\lambda) {\mathrm{diag}}(e^{-\frac{u(z,0)}{2}+\frac{u(0,0)}{2}}, e^{\frac{u(z,0)}{2}-\frac{u(0,0)}{2}}, 1)=:
\F(z,0,\lambda) D_2.$$
Then 
$$\F_-^{-1}\partial_z \F_-=\lambda^{-1}ie^{u(z,0)-\frac{u(0,0)}{2}}N_-,$$
where $N_-=\begin{pmatrix}0&0&1\\0&0&0\\0&1&0
\end{pmatrix}$ is nilpotent.
Thus 
$$\F_-(z,\lambda)=\exp(\lambda^{-1} h_-(z)N_-)$$
with $h_-(z)=i\int_0^{z} e^{u(\xi,0)-\frac{u(0,0)}{2}}d\xi=a z$, $a=i e^{\frac{u(0,0)}{2}}$,
and $\F(z,0,\lambda)=\F_-(z,\lambda) D_2^{-1}$.
Now the Iwasawa decomposition follows.

One can also observe that the extended frame $\F(z,\bar{z},\lambda)$ is of the form
 $$\F(z,\bar{z},\lambda)=\exp(\lambda^{-1}h_- N_-) D \exp(\lambda h_+ N_+), $$
 where $D={\mathrm{diag}}(w,w^{-1},1)$ with $w>0$.
Since $\exp(h_- N_-)$ is the meromorphic extended frame, we see that $h_-(z)=az$ with $a=ie^{\frac{u(0,0)}{2}}$.
Also, 
$${\mathrm{diag}(\lambda,\lambda^{-1},1)}\F(z,\bar{z},\lambda){\mathrm{diag}(\lambda^{-1},\lambda,1)}
=\exp(h_- N_-) D\exp(h_+ N_+),$$
we can just derive $\F(z,\bar{z})$ for $\lambda=1$.
Thus
$$\F(z,\bar{z})=\begin{pmatrix}
1&-\frac{a^2 z^2}{2}&az\\0&1&0\\0&az&1
\end{pmatrix}
\begin{pmatrix}w&&\\ &w^{-1}&\\ &&1\end{pmatrix}
\begin{pmatrix}1&0&0\\
\frac{h_+^2}{2}&1&h_+\\
h_+&0&1
\end{pmatrix}.$$
From $\F\in SU(3)$ we obtain 
$$w=1+\frac{|az|^2}{2}, \quad h_+=a\bar{z},$$ 
and the immersion is given by
$$f(z,\bar{z})=[(\frac{az}{1+\frac{|az|^2}{2}}, \frac{a\bar{z}}{1+\frac{|az|^2}{2}}, \frac{1-\frac{|az|^2}{2}}{1+\frac{|az|^2}{2}})].$$
Since $a$ is purely imaginary and 
$$\begin{pmatrix}
-\frac{i}{\sqrt{2}}&-\frac{i}{\sqrt{2}}&0\\
\frac{1}{\sqrt{2}}&\frac{1}{\sqrt{2}}&0\\
0&0&1
\end{pmatrix}
f(z,\bar{z})=\frac{1}{1+\frac{|az|^2}{2}}
\begin{pmatrix}
|a|\frac{z+\bar{z}}{\sqrt{2}}\\
i|a|\frac{\bar{z}-z}{\sqrt{2}}\\
1-\frac{|az|^2}{2}
\end{pmatrix},$$
the right side gives a part of surface $\R P^2$ in $\C P^2$ given by $(x_1, x_2, x_3)\in S^2$ to $[(x_1, x_2, x_3)]\in \C P^2$.
Hence $f(z,\bar{z})$ is an open subset of the surface $\R P^2$ in $\C P^2$ up to a $SU(3)$ transformation and coordinate change.
By the compactness of $SU(3)$, this holds globally. 
So Iwasawa decomposition gives rise to the embedding of $\mathbb{R}P^2$ in this case.
In addition we see that the normalized potential  is given by
$$\lambda^{-1} \begin{pmatrix}
0&0& a\\
0&0&0\\
0& a&0
\end{pmatrix} dz.$$

\subsection{Case $\psi \neq 0$, $a = b=1$}
We have seen in Section \ref{sec:vacuum} that this case  leads quite directly to a Clifford torus.
Hence $\psi\neq 0$ is a constant, $u$ is constant, and
$$\alpha=\lambda^{-1}Udz+\lambda V d\bar{z}=\lambda^{-1}\begin{pmatrix}
0&0&i\\
i&0&0\\
0&i&0
\end{pmatrix}
dz+
\begin{pmatrix}
0&i&0\\
0&0&i\\
i&0&0
\end{pmatrix}
d\bar{z},$$
with $[U,V]=0$.
Thus $\mathbb{F}=\exp(\lambda^{-1}Uz+\lambda V \bar{z})$.

\subsection{Case $\psi \neq 0$ and    $|a| \neq |b|$:}

In this last case we know the form of the normalized potential and of the Maurer-Cartan form of the extended frame associated with $\eta$.
We thus consider a  constant normalized potential $\eta$, defined on $\mathbb{C}$, of the form
\begin{equation} \label{etarad8}
\eta (z,\lambda) =  \lambda^{-1}\begin{pmatrix}
0&0&i a \\ 
i b&0&0\\
0&i a&0
\end{pmatrix}dz,
\end{equation}
where we can assume by Lemma \ref{normalized apsi} without loss of generality  $a > 0$ and $b  = -a^{-2} \psi >0$. Moreover, $|a| \neq |b|$.

From Theorem \ref{Thm8.1} and Lemma \ref{lemma8.2} we know that the metric of the corresponding entire radially symmetric minimal Lagrangian surface only depends on $r$ and the cubic form $\psi(z) dz^3 $
is constant.

We recall that considering  the function $h(s) = e^{u(r(s))}s^{\frac{1}{3}}$ with $s =  r^{\frac{3}{2}}$,
we also have  $h(s(r)) = e^{u(r)} r^{\frac{1}{2}}$. Hence 
\begin{equation} \label{omegafromh8}
u(r) = \log h(s(r)) - \frac{1}{2} \log r
\end{equation}
and $h$ satisfies
\begin{equation} \label{PIII8}
 \ddot h = \frac{(\dot h)^2}{h} - \frac{\dot{h}}{s} - \frac{16}{9} \frac{h^2}{ s}
+  \frac{16 |\psi|^2}{9} 
\frac{1}{h},
\end{equation}
with $\psi < 0$.

 Moreover, $h$  has for $s \rightarrow 0$ the asymptotic behaviour
\begin{equation} \label{PIIIasymp8}
\log (h(s) )  \approx    \frac{ 1}{3} \log s + 2\log |a|+o(s). 
\end{equation}

Finally, since $|a| \neq |b|$, Section \ref{sec:vacuum} implies that $u$ is not constant.
 
 For the treatment of the case under consideration it will be convenient to consider the frame equations directly:
\begin{equation*}
\begin{split}
\mathbb{F}^{-1}\mathbb{F}_z&=
\frac{1}{\lambda}\left(
   \begin{array}{ccc}
     0 & 0 & i e^{\frac{u}{2}} \\
  -i \psi e^{-u}   & 0 & 0 \\
     0 &  i e^{\frac{u}{2}} & 0 \\
   \end{array}
 \right)+\left(
   \begin{array}{ccc}
   \frac{u_z}{2}  & &  \\
      & -\frac{u_z}{2} &  \\
      & & 0 \\
   \end{array}
 \right)\\
 &:=\lambda^{-1}U_{-1}+U_0,\\
\mathbb{F}^{-1}\mathbb{F}_{\bar{z}} &=\lambda \left(
   \begin{array}{ccc}
     0 &  -i\bar{\psi} e^{-u} & 0 \\
     0& 0   & i e^{\frac{u}{2}} \\
     i e^{\frac{u}{2}} &0 & 0 \\
   \end{array}
 \right)+\left(
   \begin{array}{ccc}
     -\frac{u_{\bar z}}{2} & &  \\
      & \frac{u_{\bar z}}{2}  & \\
     & & 0 \\
   \end{array}
 \right)\\
  &:=\lambda V_{1}+V_0.
\end{split}
\end{equation*}

We note that $\psi$ is any  negative constant. 
We also know that $e^{\frac{u}{2}}$ only depends on $r$. Since this function also is real analytic, we can write $e^{\frac{u (z, \bar z)}{2}} = q(z \bar{z}) = q(r^2),$ where $q$ is an entire function.
As a consequence, $u_z(z, \bar z) = \partial_z u(z, \bar z) = 2 q (z\bar{z})^{-1}\partial_z q(z \bar{z}) = 
2 q(r^2)^{-1}(\partial_r q)(r^2) \bar{z}$.

Next we rewrite the frame equations above in terms of polar coordinates.
Then
\begin{align*}
\F^{-1} d\F &= (\lambda^{-1}U_{-1}+U_0)dz + \lambda V_{1}+V_0) d \bar{z} \\
&=((\lambda^{-1}U_{-1}+U_0) e^{i \theta} + (\lambda V_{1}+V_0) e^{- i \theta}) dr
+ Yd\theta.
\end{align*}

In the case under consideration we know 
$$\F(e^{i \theta}r, \overline{e^{i \theta}r}, \lambda) = 
\F(r,r,e^{-i\theta}\lambda).$$

As a consequence, $\F_0(z,\lambda) : = \F(z, \bar z, \lambda)$  satisfies
$\F_0(e^{i \theta}r, \lambda) = \F_0(r, e^{-i\theta} \lambda).$
Putting $\mu =  e^{-i\theta} \lambda$, we observe
\begin{equation*}
\F_0(r, \mu)^{-1} \partial_r \F_0(r, \mu) = ((\mu^{-1}U_{-1}+ e^{i\theta}U_0) + (\mu V_{1}+e^{-i\theta}V_0)) 
\end{equation*}
and we also know $\F_0(0,\mu) = I$.
We would like to point out that the latter condition makes sense in our case, since the function $q$ is entire.

Finally, using $u_z(z, \bar z) = 2 q(r^2)^{-1}(\partial_r q)(r^2) \bar{z}$, it is straightforward to  verify that
$e^{i\theta}U_0 + e^{-i\theta}V_0 = 0$ holds.
 
 \begin{theorem}
 Let $\eta$ be a constant potential of the form \eqref{etarad8} with $|a| \neq |b|$ and $\psi \neq 0$. Then the metric is not constant.  Moreover, the function $h(s)$ defined in 
 \eqref{omegafromh8} is not a multiple of $s^{\frac{1}{3}}$ and satisfies \eqref{PIII8} and  \eqref{PIIIasymp8}.

 Let $\F$ be the  extended frame of the associated minimal Lagrangian surface $f$
 satisfying $\F(0,0,\lambda) = I$.
 Then with the notation introduced just above we have  $\F_0(e^{i \theta}r, \lambda) = \F_0(r, e^{-i\theta} \lambda)$ and $\F_0$ satisfies the ODE
 \begin{equation}
 F_0(r, \mu)^{-1} \partial_r \F_0(r, \mu) = (\mu^{-1}U_{-1} + \mu V_{1}) dr
 \end{equation}
 with initial condition $\F_0(0,\mu) = I$.
 
 Conversely, choose any solution to  (\ref{PIII8}) and  (\ref{PIIIasymp8}) which is not a multiple of $s^{\frac{1}{3}}$. 
 Define $u$ by  (\ref{omegafromh8}).
 Then $e^{u}$ is real analytic and is of the form $e^{u (z, \bar z)} = q(r^2)$, where $q$ is entire.
 Next form the matrices $U_{-1}$ and $V_1$ as above and let $\mathbb{H}_0$ denote the solution to 
 $\mathbb{H}_0(r, \mu)^{-1} \partial_r \mathbb{H}_0(r, \mu) = (\mu^{-1}U_{-1} + \mu V_{1}) dr$ with initial condition  $\mathbb{H}_0(0, \mu) = I$.
 
 On the other hand, form the differential one-form 
$\F^{-1} d\F = (\lambda^{-1}U_{-1}+U_0)dz + (\lambda V_{1}+V_0) d \bar{z} $.
Then this PDE is solvable, since $h$ solves PIII. Moreover, 
the solution to this PDE with initial condition $\F(0,0,\lambda)=I$ is the extended frame of some minimal Lagrangian surface.
Moreover, $\F(r,r,\lambda) = \mathbb{H}_0(r,\lambda)$ and 
$\F(z, \bar z, \lambda) = \mathbb{H}_0( r, e^{-i\theta}\lambda)$.
 
 \end{theorem}
 \begin{proof}
 It suffices to observe that $\F(r,r,\mu)$ solves $ F_0(r, \mu)^{-1} \partial_r \F_0(r, \mu) = (\mu^{-1}U_{-1} + \mu V_{1}) dr$ with initial condition $\F(0,0,\mu) = I$.
 \end{proof}

Remark. (1) The frame $\F$ above yields the minimal Lagrangian surface $f = [\F.e_3]$
in $\C P^2$. Our description hence is in some sense explicit.

(2) If $h$ is a multiple of $s^{\frac{1}{3}}$, then the metric function $e^u$ is constant which implies $|a| = |b|$, as can be checked directly using \eqref{PIII8}.

Substituting $h=const \cdot s^{\frac{1}{3}}$ into \eqref{PIII8}, we see that it is a solution iff  $|\psi|=1$. From \eqref{PIIIasymp8}, the constant is $|a|^2$. Then by the assumption $h=e^u s^{\frac{1}{3}}$, we see that $e^{u}=|a|^2$. Moreover, with our assumption $a>0$ and $\psi<0$, $b=-a^{-2}\psi>0$ we get $\psi=-1$, $b=a^{-2}$. Now the equation $u_{z\bar z}=e^{-2u}|\psi|^2-e^{u}$ leads to $0=a^{-4}-a^2$ and $a=1$. Hence $b=a=1$.

(3) The solutions to entire minimal Lagrangian surfaces discussed in the 
theorem above have not been investigated so far in the sense of which are elliptic or otherwise known.
It would be interesting to understand, where in the large family of \emph{special functions} these
solutions fit in.

\begin{acknow}
This work was done mostly during the first named author's visits at Tsinghua University
and the second named author's visit at the Technical University of Munich.
The authors are grateful to both institutions for their generous support and hospitality.
The authors also would like to thank Martin Guest and Robert Conte for  helpful discussions.
The second author was supported by NSFC grants
(Grant No.~11831005, No.~11961131001 and No.~11671223).
\end{acknow}


\end{document}